\documentclass[11pt,a4paper]{article} 

\usepackage{amsfonts, amsmath, wasysym}
\usepackage{stmaryrd} 

\usepackage{tikz}
\usetikzlibrary{shapes, positioning, patterns}
\usepackage{pgfplots}
\usepgfplotslibrary{fillbetween}

\usepackage{amssymb,amsthm,
paralist
}

\usepackage{
latexsym,
}

\usepackage[margin={10pt}, font={it}]{caption} 

\usepackage{url}

\definecolor{darkgreen}{rgb}{0,0.5,0}
\definecolor{darkred}{rgb}{0.7,0,0}
\usepackage[colorlinks, 
citecolor=darkgreen, linkcolor=darkred
]{hyperref}

\usepackage[a4paper, margin=1in]{geometry}

\theoremstyle{plain}

\numberwithin{equation}{section}

\newcommand{\h}{\ensuremath{{\mathcal H}}}
\newcommand{\s}{\ensuremath{{\mathcal S}}}

\newcommand{\ca}{\ensuremath{{\mathcal A}}}

\newcommand{\cd}{\ensuremath{{\mathcal D}}}

\newcommand{\ci}{\ensuremath{{\mathcal I}}}
\newcommand{\cj}{\ensuremath{{\mathcal J}}}

\newcommand{\cv}{\ensuremath{{\mathcal V}}}

\newcommand{\zb}{\bar{z}}

\newcommand{\pl}[2]{{\frac{\partial #1}{\partial #2}}}

\newcommand{\ti}{\tilde}
\newcommand{\al}{\alpha}
\newcommand{\be}{\beta}
\newcommand{\ga}{\gamma}
\newcommand{\Ga}{\Gamma}
\newcommand{\de}{\delta}

\newcommand{\Om}{\Omega}
\newcommand{\ka}{\kappa}

\newcommand{\si}{\sigma}

\newcommand{\Si}{\Sigma}
\newcommand{\Tau}{{\mathcal T}}  
\renewcommand{\th}{\theta}

\newcommand{\vph}{\varphi}
\newcommand{\ep}{\varepsilon}

\newcommand{\R}{\ensuremath{{\mathbb R}}}
\newcommand{\N}{\ensuremath{{\mathbb N}}}

\newcommand{\C}{\ensuremath{{\mathbb C}}}

\newcommand{\downto}{\downarrow}
\newcommand{\upto}{\uparrow}
\newcommand{\embed}{\hookrightarrow}

\newcommand{\lap}{\Delta}

\newcommand{\intersect}{\cap}

\newcommand{\beq}{\begin{equation}}
\newcommand{\beql}[1]{\begin{equation}\label{#1}}
\newcommand{\eeq}{\end{equation}}
\newcommand{\beqa}{\begin{equation}\begin{aligned}}
\newcommand{\eeqa}{\end{aligned}\end{equation}}
\newcommand{\brmk}{\begin{rmk}}
\newcommand{\ermk}{\end{rmk}}
\newcommand{\partref}[1]{\hbox{(\csname @roman\endcsname{\ref{#1}})}}
\newcommand{\half}{\frac{1}{2}}

\usepackage{soul}

 \newtheorem{thm}{Theorem}[section]
\newtheorem{cor}[thm]{Corollary}
\newtheorem{lem}[thm]{Lemma}

\newtheorem{defn}[thm]{Definition}
\newtheorem{rmk}[thm]{Remark}

\newtheorem{claim}[thm]{Claim}

\usepackage{varwidth}

\title{\sc the harnack inequality without convexity for curve shortening flow}
\author{Arjun Sobnack and Peter M. Topping}
\date{20 January 2026}
\pgfplotsset{compat=1.18}

\begin{document}

%

\parskip 8pt
\parindent 0pt

\maketitle

\begin{abstract}
In 1995, Hamilton introduced a Harnack inequality for convex solutions
of the mean curvature flow. 
In this paper we prove an alternative Harnack inequality for curve shortening flow, 
i.e.~one-dimensional mean curvature flow, that does not require
any assumption of convexity. For an  initial proper curve in the plane 
whose ends are radial lines but which is otherwise arbitrarily wild, we use the Harnack inequality to give an explicit time by which the curve shortening flow evolution must become graphical. This gives a new instance of delayed parabolic regularity. 
The Harnack inequality also gives estimates describing how a polar graphical flow with radial ends settles down to an expanding solution.
Finally, we relate our Harnack inequality to Hamilton's by identifying 
a pointwise curvature estimate implied by both Harnack inequalities in the 
special case of convex flows.
\end{abstract}

\section{Introduction}
\label{intro}

Given an open interval $\ci\subset\R$, e.g.~$\ci=\R$, 
a smooth map $\ga:\ci\times(0,T)\to\R^2$, parametrised as $\ga(u,t)$ and 
defining a one-parameter family of embeddings $t\mapsto \ga(\cdot,t)$,
is said to satisfy curve shortening flow if
\beql{non_normal_CSF}
 \Big\langle \pl{\ga}{t}, \nu \Big\rangle = - \ka,  
\eeq
where $\nu$ is a unit normal to $\ga(\cdot,t)$ and $\ka$ is the corresponding geodesic curvature so that the geodesic curvature vector is $\vec{\ka}=-\ka \nu$.

If at some time $t$ the parameter $u$ is arc-length, we
have tangent vector $\Tau:=\ga_u$. 
The (continuous) tangent angle $\psi$ is determined up to an integer  multiple of $2\pi$  by the condition 
\beql{Tau_def}
\Tau=(\cos\psi, \sin\psi).
\eeq
We adopt the sign convention that $\ka:=\psi_u$ (still for arc-length $u$) which, by differentiating \eqref{Tau_def}, makes 
$$\vec{\ka}=\ga_{uu}=\Tau_u=\ka (*\Tau),$$
where the Hodge star operator $*$ gives an anticlockwise rotation to vectors,
so $\nu=-{*\Tau}$ or equivalently $\Tau=*\nu$.

In the case of non-trivial solutions
that are convex, which we may take to mean that $\ka(\cdot, t) > 0$ for all $t \in (0, T)$,
and that are parametrised so that $\pl{\ga}{t}=-\ka\nu$, Hamilton's Harnack inequality \cite{Ham_MCF_Harnack}, restricted to the case of evolving curves, tells us that 
\beql{ham_harn_1d}
\ka_t+\frac{\ka}{2t}\geq \frac{\ka_u^2}{\ka},
\eeq
still with arc-length $u$,
under reasonable conditions on the ends of the curves $\ga(\cdot,t)$.
Hamilton noticed that if one throws away the non-negative right-hand side of \eqref{ham_harn_1d} then this implies that $(\ka t^\half)_t\geq 0$, which can be integrated in time to give pointwise curvature estimates. 

In fact, a stronger inequality of this form arises by instead parametrising the flow by the tangent angle $\psi$ rather than $u$. As we show in Section  \ref{Harnack_relating_sect}, Hamilton's Harnack inequality is then \textit{equivalent} to the statement $(\ka t^\half)_t\geq 0$, where the time derivative is now with $\psi$ fixed rather than $u$ fixed.
This fact is implicit in the work of Andrews \cite[Theorem 5.6]{andrews_mathz}.
Everything we have said so far generalises to convex solutions of mean curvature flow.

In this paper we develop an alternative Harnack inequality for curve shortening flow, without any convexity hypothesis.
This was inspired by developments in \cite{ST1}, which in turn evolved from the theory of Ricci flow 
\cite{TY1} and K\"ahler Ricci flow (e.g.~\cite{DN_L}). 
In doing so we make connection with earlier work of
Neves in Lagrangian mean curvature flow \cite{neves}, cf.~Remark \ref{neves_rmk}.
This alternative Harnack inequality is related to Hamilton's by implying the same pointwise curvature estimate in the case of convex flows. However, our main application is to understand how rapidly an evolving curve straightens out. For example, we give a precise time by which a curve such as that in Figure \ref{fern_fig} becomes graphical.

\begin{figure} 
\centering
\includegraphics[width=\textwidth]{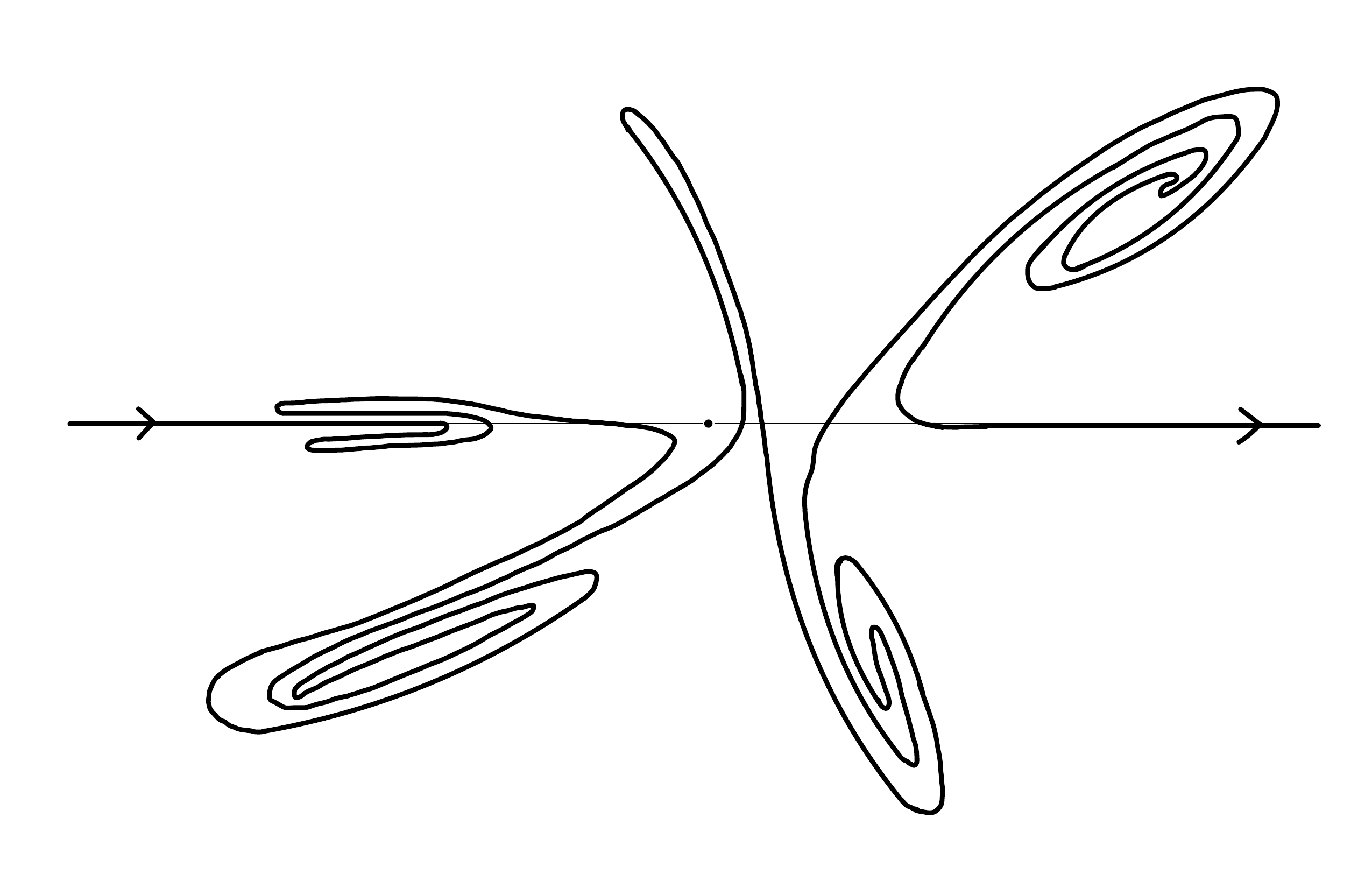}
\caption{How long before the evolution of this curve becomes graphical? Theorem \ref{graphical_cor} gives a precise time.}
\label{fern_fig}
\end{figure}

Because we are considering curve shortening flow on the noncompact domain $\ci$ rather than $S^1$, we have to be very careful with the notion of solution we take in order to avoid pathological examples of nonuniqueness and to have access to basic properties such as the avoidance principle. Certainly we ask that each curve $\ga(\cdot,t)$ is proper, so both ends of the curve end up at spatial infinity, but this is not enough. As explained by Peachey \cite{peachey_thesis}, we additionally need that the space-time map $\ga$ is proper in the following precise sense.

\begin{defn}[Uniformly proper CSF]
Let $\ci\subset \R$ be an open (parametrisation) interval.
Let $\cj\subset \R$ be a time interval of the form $[0,T)$ or $(0,T)$, where 
$T\in (0,\infty]$.
A  map $\ga:\ci\times \cj\to\R^2$ is said to be a \emph{uniformly proper} solution to curve shortening flow if 
\begin{compactenum}
\item
$\ga$ is continuous,
\item
$\ga$ is smooth on $\ci\times (0,T)$,
\item
$\ga$ is proper on $\ci\times [a,b]$ for any $[a,b]\subset \cj$, 
\item
for each $t\in \cj$ the curve $\ga(\cdot,t)$ is embedded, and
\item
$\ga$ is a solution to the curve shortening flow equation \eqref{non_normal_CSF} 
for $t\in (0,T)$.
\end{compactenum}
\end{defn}

For much of this paper we consider solutions with ends that are initially radial lines.
For $a\in\R$, we denote by $L_a$ the half-line from the origin outwards in a direction $a$. More precisely, identifying $\C\simeq\R^2$, we define
$$L_a:=\{re^{ia}\ :\ r\geq 0\}.$$

\begin{defn}
\label{CSF_IRE_def_new}
A curve shortening flow with \emph{initially radial ends at angles $a$ and $b$} is defined to be a 
uniformly proper solution $\ga:\ci\times [0,T)\to\R^2$ ($\ci\subset\R$ an open interval) to curve shortening flow
with the property that there exist $u_1,u_2\in \ci$ with $u_1<u_2$ so that 
$\ga(u,0)$ lies within $L_a$ for $u\in\ci\intersect (-\infty,u_1)$, while
$\ga(u,0)$ lies within  $L_b$ for $u\in\ci\intersect (u_2,\infty)$.
\end{defn}

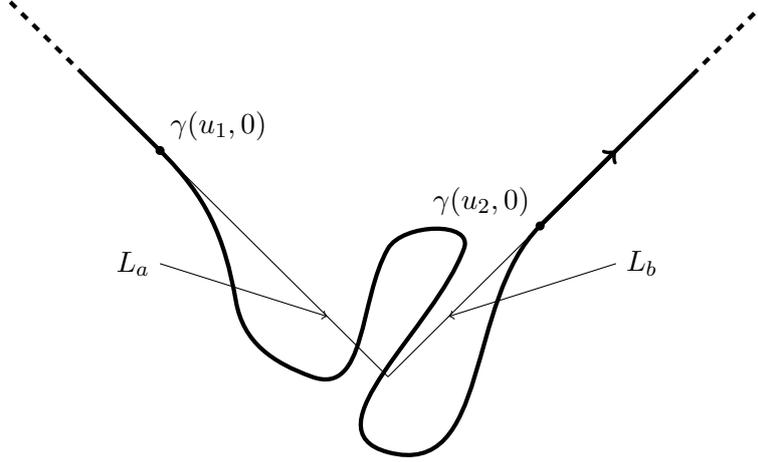
\begin{figure}
\centering
\begin{tikzpicture}[scale=1]

\draw[ultra thick, -] (-3,3) to[out=315,in=100] (-2,1) to[out=180+100, in=160] (-1, 0) to[out=160+180, in=240] (0,1.7) to[out=240+180, in=70] (1,1.7) to[out=180+70, in=165] (0,-1) to[out=165+180,in=225] (2,2) ;

\draw[-] (-3,3) node[above right]{$\ga(u_1,0)$} -- (0,0);
\filldraw (-3,3) circle (1.5pt);

\draw[->, ultra thick] (2,2) -- (3,3);
\draw[-, ultra thick] (2,2) -- (4,4);
\draw[-, dashed, ultra thick] (-4,4) -- (-5,5);

\draw[-] (0,0) -- (2,2) node[above left]{$\ga(u_2,0)$};
\filldraw (2,2) circle (1.5pt);

\draw[-, ultra thick] (-3,3) -- (-4,4);
\draw[-, dashed, ultra thick] (4,4) -- (5,5);



\draw[->] (3,1.5) node[right]{$L_b$} -- (0.8,0.8);

\draw[->] (-3,1.5) node[left]{$L_a$} -- (-0.8,0.8);

\end{tikzpicture}
\caption{Image of the initial data $\ga(\cdot,0)$ of a curve shortening flow $\ga$ with initially radial ends at angles 
$a=\frac{3\pi}{4}$ and $b=\frac{\pi}{4}$.}
\label{CSF_IRE}
\end{figure}

The definition is illustrated in Figure \ref{CSF_IRE}.

Owing to the work of Polden \cite{Polden}, Huisken \cite[Theorem 2.5]{huisken_AJM} and Chou-Zhu \cite{CZ},  given any proper  locally Lipschitz embedded curve $\ga_0:\ci\to\R^2$ with radial ends in the sense as described in Definition \ref{CSF_IRE_def_new} for $\ga(\cdot,0)$, there exists a unique 
uniformly proper solution $\ga:\ci\times [0,\infty)\to\R^2$ for which 
$\ga(\cdot,0)=\ga_0$. 
It is implicit in the literature that the flow continues to be asymptotic to the lines $L_a$ and $L_b$ at later times. 
In particular it continues to be possible to write the ends of this solution as graphs over ends of the lines $L_a$ and $L_b$, at least over compact time intervals;
see \cite[Lemma 2.5.1]{arjun_thesis}. Later we will prove refined estimates concerning the exponential convergence of the evolving curves to $L_a$ and $L_b$ as we move down the ends at later times (see Appendix \ref{app:expdec}). 

To state the alternative Harnack inequality, and its applications, we need the following notion of signed area,
which we illustrate in Figure \ref{swept_area_fig}.
\begin{defn}
\label{swept_area_def}
Given a smooth immersed curve $\ga:\ci\to\R^2$, with $\ci\subset \R$ an open interval, we define the 
\emph{swept area} $\ca(v,w)$, for $v,w\in\ci$ with $v\leq w$, to be 
\beql{swept_area}
\ca(v,w):=\half\int_{v}^{w} *(d\ga\wedge \ga)
=\half\int_{v}^{w} *(\ga_u\wedge \ga) \, du,
\eeq
where $*$ is the Hodge star operator on 2-vectors in $\R^2$, i.e.
$*\colon\Lambda^2\R^2 \to\R$ is determined by $*(e_1\wedge e_2)=1$.
\end{defn}

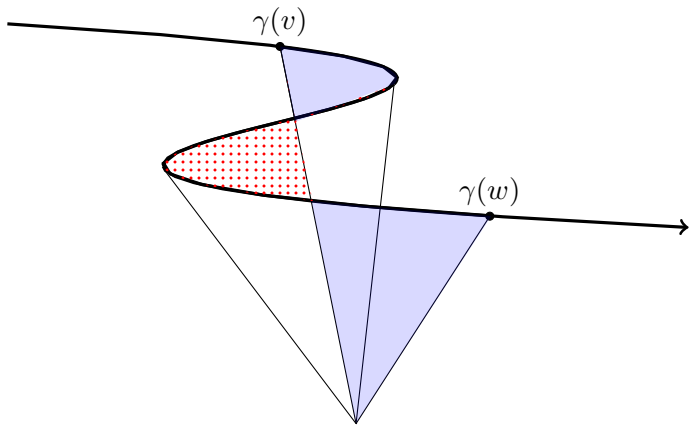
\begin{figure}
\centering
\begin{tikzpicture}[scale=1]

\draw[very thick, samples=20,domain=-1.3:1.4, ->] 
plot ({4*\x*(\x-1)*(\x+1)},-\x);

\draw[very thick, samples=20,domain=-1:-0.5, name path=f] 
plot ({4*\x*(\x-1)*(\x+1)},-\x);

\draw[very thick, samples=20,domain=1.048:1.248, name path=h] 
plot ({4*\x*(\x-1)*(\x+1)},-\x);

\draw[very thick, samples=20,domain=-0.5:1.05, name path=minus] 
plot ({4*\x*(\x-1)*(\x+1)},-\x);

\draw[-] (0,1) node[above]{$\ga(v)$} -- (1,-4);

\filldraw (0,1) circle (1.5pt);

\draw[-, name path=g] (0,1) -- (0.2,0);

\draw[-] (1.5,0.5) -- (1,-4);

\draw[-] (-1.55,-0.6) -- (1,-4);

\draw[-] (2.76,-1.25) node[above]{$\ga(w)$}-- (1,-4);

\filldraw (2.76,-1.25) circle (1.5pt);

\draw[-, name path=origin] (1,-4) -- (1,-4);

\tikzfillbetween[of=f and g]{blue, opacity=0.15}

\tikzfillbetween[of=h and origin]{blue, opacity=0.15} 

\tikzfillbetween[of=minus and g]{red, opacity=1, pattern=dots, pattern color=red}

\end{tikzpicture}
\caption{The swept area $\ca(v,w)$ is the blue shaded area minus the red dotted area.}
\label{swept_area_fig}
\end{figure}

The swept area is invariant under rotations of $\ga$ about the origin, and negates under reflections across lines through the origin.
Suitably interpreted, it is also invariant under orientation-preserving reparametrisations of $\ga$, and negates under orientation-reversing reparametrisations. 

In the case that $u$ is an arc-length parameter, we can write
\beql{simpler_A_exp}
\ca(v,w)=-\half\int_{v}^{w} \langle \ga,\nu \rangle \, du.
\eeq
If we view the curve as lying within the complex plane $\C\simeq \R^2$, then we can alternatively write  
\beql{complex_A_exp}
\ca(v,w)=\Im\left[-\frac{1}{2}\int \zb\, dz\right],
\eeq
where the integral is performed on the part of $\ga$ parametrised between $v$ and $w$, and $\Im$ represents the imaginary part.

\begin{rmk}
\label{ca_for_closed_curve}
%
If we are given an immersed closed (piecewise) smooth curve $\ga:S^1\to\R^2$, 
then the total swept area can be written
$$\ca=\half\int_\ga *(\ga_u\wedge \ga) \, du=-\frac{1}{2i}\int_\ga \zb\, dz,$$
because the real part of the integral in \eqref{complex_A_exp} vanishes.
If additionally $\ga$ is an embedding that traverses its image in an  anti-clockwise direction (i.e.~it has winding number $0$ or $+1$ around each point), then the total swept area will be the \textbf{negation} of the area of the region enclosed by $\ga$.
The choice of sign convention for the swept area is fraught because we are using it to relate two different topics in which the standard sign conventions clash.
\end{rmk}
We write the swept area of an evolving curve $\ga(\cdot,t)$ as $\ca(v,w,t)$.

Before giving the Harnack inequality, we give one of the applications.
Given a curve shortening flow with initially radial ends, but which can be arbitrarily wild otherwise, as in Figure \ref{fern_fig}, we obtain an explicit sharp time after which we can be sure that the flow can be written as a graph globally, with an explicit gradient bound. From that moment on, parabolic regularity theory is switched on and we obtain $C^k$ estimates 
for the flow in terms of the swept area at time zero and the time for which we have flowed.

\begin{thm}
\label{graphical_cor}
Let  $\ga:\ci\times [0,T)\to \R^2$, with $\ci\subset \R$ an open interval, be a curve shortening flow with initially radial ends at angles $\frac\pi2 + \frac\be2$ and $\frac\pi2- \frac\be2$ in the sense of Definition \ref{CSF_IRE_def_new}, where $\be\in (0,\pi]$,
which is smooth down to $t=0$.
Suppose that the swept area of the initial curve has bounds 
\beql{ca_bds}
\ca_-\leq \ca(v,w,0)\leq \ca_+
\eeq
for all $v,w\in\ci$ with $v\leq w$.
Then 
for $t>\frac{2(\ca_+-\ca_-)}{\be}$, i.e.~for $t > 0$ satisfying $\frac{\ca_+ - \ca_-}{t} < \frac\be2$, 
\begin{enumerate}
    \item $\ga$ can be written as the graph of a function $y$ (over the $x$--axis) which satisfies 
    \beql{graphbdd} \| y_x(\cdot, t)\|_{L^\infty} \leq \tan\left[ \frac\pi2 - \frac\be2 + \frac{\ca_+ - \ca_-}{t} \right] \eeq
    \item the portion of $\ga(\cdot, t)$ lying within the double sector  consisting of all points $r e^{i\th}\in\C\simeq\R^2$ with $r\in \R$ and 
$$\th\in \s:=\left( \frac\pi2 - \frac\be2 + \frac{\ca_+-\ca_-}{t}, \frac\pi2 + \frac\be2 -\frac{\ca_+-\ca_-}{t}\right)$$
can be written as a polar graph, i.e.~as 
$\{r(\th) e^{i \th}\ :\ \th\in\s\}$ for some 
$r:\s:\to\R$. 
\end{enumerate}
\end{thm}

Note that the initial swept area $\ca(v,w,0)$ is automatically bounded because the integrand in its definition \eqref{swept_area} vanishes on the straight ends, and thus has compact support. Because $\ca(v,v,0)=0$, we must have $\ca_-\leq 0$ and 
$\ca_+\geq 0$.

Before giving any further consequences of the Harnack inequality, we explain the Harnack inequality itself. In order to do that, we need one further definition.

\begin{defn}
\label{Psi_def}
Given a smooth immersed curve $\ga:\ci\to\R^2$, with $\ci\subset \R$ an open interval, 
if we make a continuous choice $\psi:\ci\to\R$ of tangent angle as in 
\eqref{Tau_def},
then we  define the \emph{turning function} $\Psi(v,w)$, for $v,w \in \ci$ with $v \leq w$, to be
$$\Psi(v,w):=\psi(w)-\psi(v).$$
\end{defn}

The turning function $\Psi$ is well-defined, independently of the normalisation chosen for $\psi$, and possesses the same symmetries as the swept area. 

We write the turning function of an evolving curve $\ga(\cdot,t)$ as $\Psi(v,w,t)$.
The central estimate exploited in the proof of Theorem \ref{graphical_cor}
is on this turning function:
\beql{Psi_goodest}
-\frac{\ca_+ - \ca_-}t \leq \Psi(v, w, t) \leq [\pi - \be] + \frac{\ca_+ - \ca_-}t
\eeq
for all $v,w \in \ci$ with $v \leq w$ and all $t \in (0, T)$.
The Harnack inequality is one of the ingredients needed to show \eqref{Psi_goodest}.

\begin{thm}[Alternative Harnack]
\label{main_new_harnack_thm}
Let  $\ga:\ci\times [0,T)\to \R^2$, with $\ci\subset \R$ an open interval, be a curve shortening flow with initially radial ends 
in the sense of Definition \ref{CSF_IRE_def_new}, 
which is smooth down to $t=0$.
Suppose that the swept area of the initial curve has bounds 
\beql{harn_bds}
\ca_-\leq \ca(v,w,0)\leq \ca_+
\eeq
for all $v,w\in\ci$ with $v\leq w$.
Then the Harnack quantity
\beql{gH_def}
\h(v,w,t):=\ca(v,w,t)-t\Psi(v,w,t)
\eeq
satisfies the  Harnack inequality
\beql{gen_Harn}
\ca_-\leq \h(v,w,t)\leq \ca_+
\eeq
for all $v,w\in\ci$ with $v<w$ and all $t\in [0,T)$.
\end{thm}
As already mentioned, this result is related to previous work, particularly from
\cite{ST1, TY1, DN_L, neves}.
We comment further on some of these connections in Sections 
\ref{evol_eq_sect} and \ref{gen_H_pf} where we prove 
Theorem \ref{main_new_harnack_thm}.

We will prove the graphicality of Theorem \ref{graphical_cor} by controlling the 
tangent angle $\psi(u,t)$ using the  Harnack inequality of Theorem \ref{main_new_harnack_thm}. However, in order to turn the information 
\eqref{gen_Harn}
on $\h$ for $t>0$ into information on $\Psi$ and $\psi$ at the same time $t>0$, 
we need good control on the swept area 
$\ca(\cdot,\cdot,t)$ for that $t>0$. 
The evolution equation \eqref{ca_eq} for $\ca$ that we derive later in Lemma \ref{AH_eqs_lem} tells us that the swept area 
$\ca(v,w,t)$ can increase at an arbitrarily large rate if the curve spirals a large number of times between $v$ and $w$, so initially it looks hopeless to obtain strong enough control on $\ca(\cdot,\cdot,t)$. However, it turns out to be possible to rule out this bad spiralling behaviour between $v$ and $w$ when $(v,w)$ is an extremal of $\ca(\cdot,\cdot,t)$. This leads to the following control, which is key to extracting applications from the Harnack inequality.

\begin{thm}
\label{ca_control_thm}
In the setting of Theorem \ref{main_new_harnack_thm}, 
if,  after rotating, the initially radial ends are at  angles $\pi$ and 
$\al \in [0, \pi)$,
then  
\beql{ca_control_thm_equation}
\ca_-\leq \ca(v,w,t)\leq \al t + \ca_+,
\eeq
for all $v,w\in\ci$ with $v\leq w$, and all $t\in [0,T)$.
\end{thm}
When coupled with the alternative Harnack estimate Theorem \ref{main_new_harnack_thm}, we get the effective control \eqref{Psi_goodest} on the turning function required to prove Theorem \ref{graphical_cor}.
The proofs of Theorems  \ref{graphical_cor} and \ref{ca_control_thm} are given in Section \ref{arjun_sect}.

\begin{figure}
\centering
\begin{tikzpicture}[scale=1]

\draw [thick, ->] (0,0) -- (-2.8*1,2.8*1.732);
\draw [thick, ->] (0,0) -- (-2.8*2,0);

\draw [-, dotted] (0,0) -- (-2.8*0.985*2, 2.8*0.174*2 );
\draw [-, dotted] (0,0) -- (-2.8*0.866*2, 2.8*0.5*2 );

\draw[very thick, samples=30,domain=0.5:58.8, ->] 
plot ({-cos(\x)*(\x/80+1.5-0.65*cos(6*\x))/(sin(3*\x))^(0.5)}, {sin(\x)*(\x/80+1.5-0.65*cos(6*\x))/(sin(3*\x))^(0.5)} ) node[right]{$\ga(\cdot,t)$};

\draw[-] (-0.65,0.22) node[right]{$\be$};
\draw[samples=40,domain=0:60] plot ({-cos(\x)/1.5}, {sin(\x)/1.5}); 

\end{tikzpicture}
\caption{Illustration of a solution with $\ca_-=0$.
Each dotted radial line intersects the curve $\ga(\cdot, t)$ precisely once.}
\label{polar_graphical_case}
\end{figure}
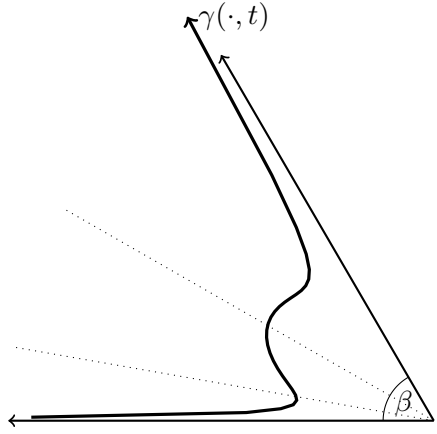

Consider the case $\ca_-=0$.
Intuitively this means that $\ga(\cdot,0)$ cannot move anti-clockwise around the origin.
Keeping in mind the embeddedness of $\ga(\cdot,0)$,
which prevents the curve from passing through the origin
(except in the trivial case that the flow is a static straight line through the origin)
we see that in this case $\ga(\cdot,0)$, and the subsequent flow $\ga$, map into the sector between $L_\pi$ and $L_{\al}$
defined by $S = \{re^{i\th}\ :\ r\geq 0,\ \th\in [\al,\pi]\}\subset \C\simeq\R^2$. 
Moreover,
$\ga(\cdot,t)$ describes a polar graph for $t\in (0,T)$ in the sense given in \cite[Section 3]{ST1}, as illustrated in Figure \ref{polar_graphical_case}.
Indeed, by considering Angenent's crossing principle (see e.g.~\cite{angenent1991parabolic, ST1}) $\ga(\cdot,t)$ will meet each line $L_\th$ with $\th\in (\al,\pi)$ exactly once for $t\in (0,T)$.
In order that we parametrise our curves in the same direction as before, we use a polar angle $\vph:=\pi-\th$, where $\th$ is the usual polar angle measuring anticlockwise from the positive $x$ axis. 
In this way, we parametrise the curves by $\vph\in (0,\be)$, where we interpret $\be := \pi - \al > 0$ as the interior angle of the sector $S$.
In particular, we can write the swept area of our polar graphs $\ga$ as
$$\ca(\vph_0,\vph_1):=\half\int_{\vph_0}^{\vph_1} *(\ga_\vph\wedge \ga) \, d\vph.$$
If we write $r(\vph)$ for the radial polar value of the graph then 
$*(\ga_\vph\wedge \ga)=r^2$. 
Indeed, if we write $\xi=\xi(\vph):=(\cos\vph,\sin\vph)$ so $\ga(\vph)=r\xi$, then 
$\ga_\vph=r_\vph \xi+r ({-{*\xi}})$, and so $\ga_\vph\wedge \ga=r^2 ({-{*\xi}})\wedge\xi = r^2 \xi\wedge  (*\xi) $,
and hence $*(\ga_\vph\wedge \ga)=r^2$.
Therefore the swept area coincides with the 
area under the polar graph in the sense that 
\begin{equation}
\label{ca_cv}
\ca(\vph_0,\vph_1)=\cv(\vph_0,\vph_1):=\int_{\vph_0}^{\vph_1} \frac{r(\vph)^2}{2} \, d\vph.
\end{equation}
It is natural to extend the domain on which $\vph_0<\vph_1$ live to the whole closed interval $[0,\be]$ when the resulting total area is finite, as it is here.

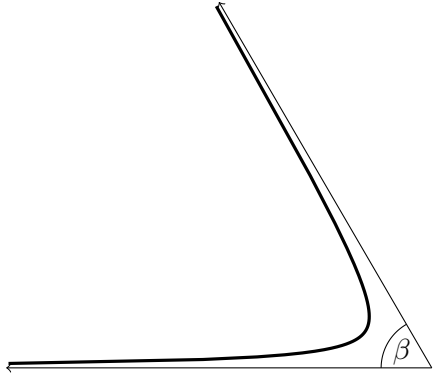
\begin{figure}
\centering
\begin{tikzpicture}[scale=1]

\draw [->] (0,0) -- (-2.8*1,2.8*1.732);
\draw [->] (0,0) -- (-2.8*2,0);

\draw[very thick, samples=40,domain=0.6:59.4] plot ({-cos(\x)/(sin(3*\x))^(0.5)}, {sin(\x)/(sin(3*\x))^(0.5)} );

\draw[-] (-0.65,0.22) node[right]{$\be$};
\draw[samples=40,domain=0:60] plot ({-cos(\x)/1.5}, {sin(\x)/1.5}); 

\end{tikzpicture}
\caption{Qualitative picture of the $\be$-wedge solution $\ga_\be$ at $t=1$.}
\label{beta_wedge_fig}
\end{figure}

A particularly natural convex polar graphical solution to curve shortening flow arises as the  solution starting with a straight line that has been bent at the origin to have an opening angle $\be\in (0,\pi)$, as illustrated in Figure \ref{beta_wedge_fig} and discussed in  \cite[Section 3]{ST1}.
This so-called $\be$-wedge solution scales self-similarly; if at each time $t > 0$ we parametrise $\ga(\cdot,t)$ and $\ka(\cdot,t)$ 
by  the tangent angle $\psi\in (0,\pi-\be)$ (or the polar angle $\vph\in (0,\be)$), then we have
\beql{gamma_beta}
t^{-\half}\ga(\psi,t)=\ga_\be(\psi):= \ga(\psi,1)
\eeq
and 
\beql{kappa_beta}
t^{\half}\ka(\psi,t)=\ka_\be(\psi):= \ka(\psi,1).
\eeq
The $\be$-wedge solution is a polar graphical curve shortening flow for which the Harnack quantity $\h$ is identically zero. This can be derived by integrating the equation for an expander, or by considering area as in \cite[Section 3]{ST1}. Thus, writing $\cv_\be(\vph_0,\vph_1)$ and 
$\Psi_\be(\vph_0,\vph_1)$ for the swept area and turning function of the $\be$-wedge solution at time $t=1$, we have
\beql{wasL71}
\cv_\be(\vph_0,\vph_1)\equiv \Psi_\be(\vph_0, \vph_1).
\eeq

\begin{figure}
\centering
\begin{tikzpicture}[scale=1.2]
\draw [->] (0,0) -- (3*1,3*1.732) node[right]{$\vph=\beta$};
\draw [->] (0,0) -- (-3*2,0) node[below]{$\vph=0$};
\draw [-] (0,0) -- (-1,1) ;
\draw [-.] (-1.3,1.05) node[above]{$\ga(\vph_0, t)$};
\fill[blue!8] (0,0) to  (.5,2.5) to[out=-260,in=180-240] (0,4) to[out=-240,in=180-160] (-1,1) to (0,0);
\draw[very thick, <-] (2.95,3*1.734) to[out=180+59,in=180-260] (.5,2.5) to[out=-260,in=180-240] (0,4) to[out=-240,in=180-160] (-1,1) to[out=-160, in=180-130] (-3,.9) to[out=-130, in=0] (-6,0.05);
\draw [-] (0,0) -- (-1,1) ;
\draw [-] (0,0) -- (0.5,2.5) ;
\draw [<-] (-0.05,3) -- (-.5,3) node[left]{$\mathcal{V}(\vph_0, \vph_1, t)$};
\draw[->] (1.05,2.95)  node[above]{$\ga(\vph_1, t)$} to (.6,2.6);
\draw[thick] [->] (.5,2.5) -- (.65,1.9) ;
\draw[thick] [->] (-1,1) -- (-.5,1.2) ;
\node at (-1,1) [circle,fill,inner sep=1pt]{};
\node at (.5,2.5) [circle,fill,inner sep=1pt]{};
\end{tikzpicture}
\caption{Illustration of $\mathcal{V}(\vph_0, \vph_1, t)$ and $ \Psi(\vph_0, \vph_1, t) $ for \( \beta \approx \frac{2\pi}{3} \). 
Pictorially,  $\Psi(\vph_0, \vph_1, t)  \approx -\frac{\pi}{2}$ is the anti-clockwise angle through 
which the arrow at $\ga(\vph_0,t)$ turns to get to the arrow at $\ga(\vph_1, t)$.}
\label{arjun_fig}
\end{figure}
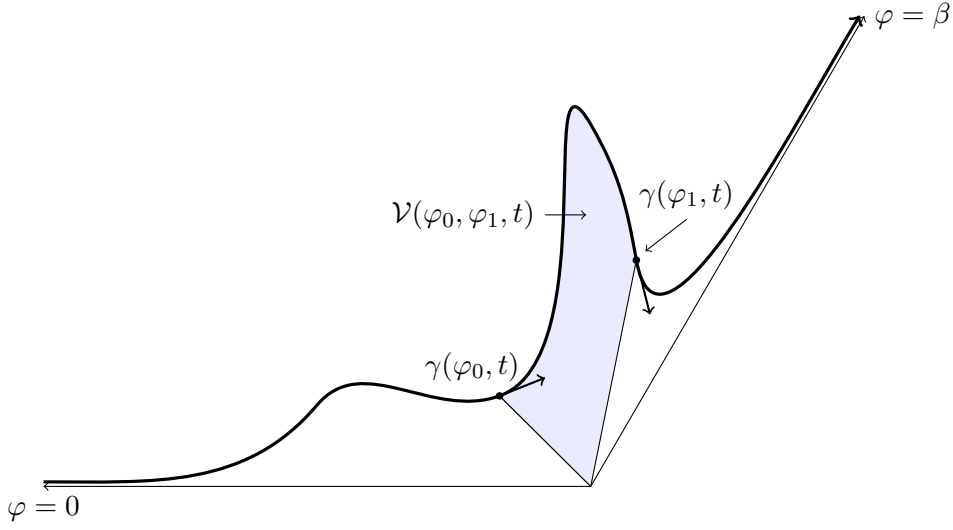

The  Harnack inequality of Theorem \ref{main_new_harnack_thm} will give us most of the following theorem,
which was originally stated in the first version of 
\cite{ST1}.

\begin{thm}[Polar graphical Harnack inequality]
\label{wedge_harnack_thm1}
Let $\beta\in (0,\pi]$, and  suppose that $r:(0,\beta)\times [0, \infty)\to (0,\infty)$, written as $r(\vph,t)$,
is a smooth function that is the radial polar coordinate of a curve shortening flow 
with initially radial ends  at angles $\pi$ and $\pi - \beta$,
that is a polar graph for $t>0$.
We define
$$ \cv(\vph_0,\vph_1,t):=\int_{\vph_0}^{\vph_1}\frac{r(\vph,t)^2}{2} \, d\vph, $$
for $0\leq \vph_0\leq \vph_1\leq \be$
(see Figure \ref{arjun_fig})
to represent the area under the polar graph,
and abbreviate $\cv(t):=\cv(0,\beta,t)$ and  
$\cv_0:=\cv(0)<\infty$.
Then for all $t\geq 0$  we have
\begin{equation}
\label{total_V}
\cv(t) = \cv_0 + (\pi - \be)t,
\end{equation}
while for all $t> 0$ and all $0\leq\vph_0\leq\vph_1\leq\beta$,  we have
\begin{equation}
\label{wh1}
0\leq \h(\vph_0,\vph_1,t):=\cv(\vph_0,\vph_1,t)-t \Psi(\vph_0, \vph_1,t)
\leq \cv_0.
\end{equation}
If $\be = \pi$, then 
\beq \label{wh2pi} 
- \frac{ \cv_0 }{t} \leq  \Psi(\vph_0, \vph_1, t) \leq \frac{\cv_0}{t} \eeq
for $t > 0$, while if 
$\be \in (0,\pi)$, then 
\begin{equation}
\label{wh2}
- \frac{ \cv_0 }{t} \leq  \Psi(\vph_0, \vph_1, t) - \Psi_\be(\vph_0, \vph_1)  \leq \frac{\cv_0}{t}
\end{equation}
for $t>0$, where $\Psi_\be(\vph_0, \vph_1)$ is the turning function of the $\be$-wedge solution.
\end{thm}

Observe that the function $r(\cdot,0)$ is parametrising the initial curve without the ends, 
i.e.~between $\ga(u_1)$ and $\ga(u_2)$ in Definition \ref{CSF_IRE_def_new}, cf.~Figure \ref{CSF_IRE},
whereas for $t>0$,
$r(\cdot,t)$ is parametrising the entire curve.

The proof of Theorem \ref{wedge_harnack_thm1} will be given in Section \ref{gen_H_pf}.

\begin{rmk}
\label{relax_ends_rmk}
Theorems \ref{graphical_cor}, \ref{main_new_harnack_thm} and \ref{wedge_harnack_thm1}  consider flows with initially radial ends. They can all be generalised to ask only that the ends are asymptotically radial in a sufficiently strong sense, e.g.~that they can be written as graphs over radial lines of functions that decay exponentially together with their derivatives. 
This latter condition is preserved under the flow, 
see Appendix \ref{app:expdec}.
\end{rmk}

\medskip

\emph{Acknowledgements:} 
AS was supported by EPSRC studentship EP/R513374/1.
PT  was partly supported by EPSRC grant EP/T019824/1.
PT was partly supported by the National Science Foundation (USA) under Grant No.~DMS-1928930, while  PT was Chern professor at the Simons Laufer Mathematical Sciences Institute in Berkeley, California, during the Autumn 2024 semester. 
PT thanks Gerhard Huisken for discussions about Hamilton's Harnack inequality. PT thanks Felix Schulze for  
discussions about the existing literature.
For the purpose of open access, the authors have applied a Creative Commons Attribution (CC BY) licence to any author accepted manuscript version arising.

\section{Evolution of the swept area, turning function and  Harnack quantity}
\label{evol_eq_sect}

In this section we derive evolution equations for the swept area $\ca$ and turning function $\Psi$, and the Harnack quantity $\h$, as introduced in Definitions \ref{swept_area_def} and \ref{Psi_def}, and in \eqref{gH_def}.
In order to derive these heat equations, we need notation for the Laplacian.
Given a curve $\ga(u)$, with parameter $u\in \ci$ and generally also evolving in time, 
if $f(u)$ is a function having $u$ as an 
argument (possibly amongst others), then we write 
$\lap_u f$ for the Laplacian of $f$ with respect to the pull-back of the Euclidean metric on $\R^2$ by $\ga$. If $u$ has been chosen as an arc-length parameter for $\ga$ at some time, then 
$$\lap_u f=f_{uu}$$
at that time.
Note that we can write curve shortening flow as 
$$\ga_t=\lap_u \ga.$$

It is well known how tangent angles evolve under curve shortening flow. For example, the following lemma is immediately implied by \cite[Lemma 3.1.5]{gage_hamilton}.
\begin{lem}
\label{th_psi_lem}
Let $\ga:\ci\times (0,T)\to \R^2$ be a smooth map 
such that $\ga(\cdot,t)$ is an evolving immersed curve satisfying the curve shortening flow equation $\ga_t=
\vec{\ka}$, where $\vec{\ka}$ is the curvature vector of $\ga(\cdot,t)$.
Let  $\psi(u,t)$ be the tangent angle and  $\Psi(v,w,t)$ be the turning function of $\ga $, cf.~Definition \ref{Psi_def}.
Then
\beql{th_eq}
\psi_t=\lap_u \psi
\eeq
and, in particular, 
$$\Psi_t=\lap_v \Psi + \lap_w \Psi$$
on $\Si\times (0,T)$.
\end{lem}

We are now in a position to describe the PDE governing the swept area $\ca$. 

\begin{lem}
\label{AH_eqs_lem}
Let $\ga:\ci\times (0,T)\to \R^2$ be a smooth map 
such that $\ga(\cdot,t)$ is an evolving immersed curve satisfying the curve shortening flow equation
$\ga_t=\vec{\ka}$, where $\vec{\ka}$ is the curvature vector of $\ga(\cdot,t)$.
Let $\ca(v,w,t)$ denote the swept area of $\ga(\cdot,t)$
between $u=v$ and $u=w$, cf.~Definition \ref{swept_area_def}.
Then
\beql{ca_eq}
\ca_t=\lap_v\ca+\lap_w\ca+\Psi(v,w,t),
\eeq
where $\Psi$ is the turning function of $\ga$.
In particular, the Harnack quantity $\h(v,w,t):=\ca(v,w,t)-t\Psi(v,w,t)$
from Theorem \ref{main_new_harnack_thm} satisfies
\beql{wedge_harnack_eq}
\h_t=\lap_v\h+\lap_w\h.
\eeq
\end{lem}

\begin{proof}[Proof of Lemma \ref{AH_eqs_lem}]
As in the proof of Lemma \ref{th_psi_lem}, 
it suffices to establish the equations at an arbitrary time $t=\tau\in (0,T)$, and we choose
$u\in\ci$ to be an arc-length parameter at time $t=\tau$.
Compute at time $t=\tau$
$$2\ca_{w}=*\big[\ga_u\wedge \ga\big](w)$$
and 
\beql{Aww}
2\ca_{ww} = *\big[\ga_{uu}\wedge \ga\big](w)
=*\big[\ga_t\wedge \ga\big](w),
\eeq
since $\ga_t=\ga_{uu}$.
Similarly we have $2\ca_{v}=-*\big[\ga_u\wedge \ga\big](v)$ and
\beql{Avv}
2\ca_{vv} = -\, {*}\big[\ga_t\wedge \ga\big](v).
\eeq
%
Because we are fixing a parametrisation $u$, rather than trying to continually adjust to arc-length parametrisation as $t$ changes as in \cite{gage_hamilton}, we have $\ga_{ut}=\ga_{tu}$ and can integrate by parts to compute
\beqa
\label{At}
2\ca_t &= *\int_{v}^w (\ga_{tu}\wedge\ga+\ga_u\wedge \ga_t) \, du\\
&= *\big[\ga_t\wedge\ga\big]_{v}^w
+*\,2\int_{v}^w \ga_u\wedge \ga_t\, du.
\eeqa
Combining \eqref{Aww}, \eqref{Avv} and \eqref{At} gives
$$\ca_t=\ca_{vv}+\ca_{ww}+
*\int_{v}^w \ga_u\wedge \ga_t\, du.$$
Recall from Section \ref{intro} that at time $t=\tau$ we have
$\ga_t = -\ka \nu = \ka (*\ga_u)$,
and so
$$*(\ga_u\wedge \ga_t)=\ka\ {*}(\ga_u \wedge [* \ga_u]  )
=\ka$$
because $\ga_u$ has unit length. 
As
\beql{psi_int}
\Psi(v,w,\tau)=\int_v^w \ka(u) \, du,
\eeq
we obtain 
\beq
\ca_t=\ca_{vv}+\ca_{ww}+\Psi(v,w,\tau),
\eeq
which is \eqref{ca_eq} at our arbitrary time $t=\tau$. 
The final conclusion \eqref{wedge_harnack_eq} then follows immediately by combining with Lemma \ref{th_psi_lem}.
\end{proof}

\begin{rmk}\label{neves_rmk}
If we fix $v_0 \in \ci$ and define $\be : \ci \times [0, T) \to \R$, up to a normalisation constant, by $\be(w, t) := - 2 \, \ca(v_0, w, t) + V(t)$, where $V : [0, T) \to \R$ is chosen to satisfy $V'(t) = -2 \psi(v_0,t) - {*[\ga_{uu} \wedge \ga]}(v_0,t)$, 
then the $\be(w,t)$ here is the $\be$ introduced in \cite[\S3.1.2]{neves}. 
\end{rmk}

\section{Proof of the alternative Harnack inequality}
\label{gen_H_pf}

In this section we apply the evolution equations derived in the previous section in an unfamiliar maximum principle argument in order to prove 
Theorem \ref{main_new_harnack_thm}.

We begin by simplifying the parametrisation $\ga$ without any loss of generality.
Since $\ga$ is a curve shortening flow with initially radial ends, and in particular is uniformly proper with bounded curvature, we can 
reparametrise so that $\ga$ satisfies 
$\ga_t= \vec{\ka}=-\ka\nu$ rather than the more general \eqref{non_normal_CSF}.
Next, we make a time-independent reparametrisation of $\ci$ to reduce to the case
that $\ci=(0,1)$. 

Define  
\beql{Si_def}
\Si:=\{(v,w)\in (0,1)\times(0,1)\ :\ v<w\}.
\eeq
In Theorem \ref{AH_eqs_lem} we have shown that $\h$ satisfies a heat equation on 
$\Si\times (0,T)$ and we will aim to apply the maximum principle to $\h$ on $\Si\times [0,T)$. This  will require us to extend $\h$ to the boundary edges in $\partial\Si\times [0,T)$ and control it there. To do that we find a new heat equation that governs $\h$ one dimension down on these boundary components and apply the maximum principle. Of course, this requires us to go yet another dimension down and understand how
$\h$ behaves on the boundary points of these one-dimensional edges, but this will result from our knowledge of the spatial asymptotics of $\ga$.

In the following lemma 
the tangent angle $\psi$ and turning function $\Psi$ are from Definition \ref{Psi_def}, but because of the initially radial end at angle $\pi$, we are free to determine the choice of $\psi$ (otherwise determined up to an integer multiple of $2\pi$) by the condition that $\psi(u,t)\to 0$ as $u\downto 0$. Of course, the limit is $0$ rather than $\pi$ because we are moving in, rather than out, along the end.
Moreover, because the curve $\ga$ in the following lemma is embedded, and it exits along the radial line at angle $\al$, we must have 
$\psi(u,t)\to\al$ as $u\upto 1$.

\begin{lem}
\label{boundary_lem}
Let $\ga:(0,1)\times [0,T)\to \R^2$ be a curve shortening flow with initially radial ends at angles 
$\pi$ and $\al\in (-\pi,\pi)$ in the sense of Definition \ref{CSF_IRE_def_new}, 
that is smooth down to $t=0$.
Suppose further that $\ga$ satisfies  
$\ga_t = \vec{\ka}$, i.e. it has no tangential reparametrisation that would be allowed by \eqref{non_normal_CSF}.

Then an appropriate choice of $\psi$ extends continuously to $[0,1]\times [0,T)$, with 
$\psi(0,t)=0$ and $\psi(1,t)=\al$. The resulting extension of 
$\Psi$ to $\overline\Sigma \times [0,T)$  satisfies $\Psi(0,w,t)=\psi(w,t)$ and
\begin{equation}
\label{Psi_props}
\left\{
\begin{aligned}
\left(\pl{}{t}-\lap_w\right)\Psi(0,w,t) &= 0 &\quad w\in (0,1) \\
\Psi(0,0,t) &= 0 &\\
\Psi(0,1,t)&= \al,&
\end{aligned}
\right.
\end{equation}
for all $t\in (0,T)$.
Moreover, $\ca$ extends continuously to 
$\overline\Sigma \times [0,T)$,
and the resulting function $\ca(0,w,t)$ satisfies
\begin{equation}
\label{A_props}
\left\{
\begin{aligned}
\left(\pl{}{t}-\lap_w\right)\ca(0,w,t) &= 
\Psi(0,w,t)&\quad w\in (0,1)\\
\ca(0,0,t) &= 0 &\\
\ca(0,1,t)&= \ca(0,1,0)+\al t&
\end{aligned}
\right.
\end{equation}
for all $t\in (0,T)$.
In particular, the Harnack quantity 
$\h(v,w,t):=\ca(v,w,t)-t\Psi(v,w,t)$ defined in Theorem \ref{main_new_harnack_thm}
extends continuously to $\overline\Sigma \times [0,T)$ and $\h(0,w,t)$ satisfies
\begin{equation}
\label{h_props}
\left\{
\begin{aligned}
\left(\pl{}{t}-\lap_w\right)\h(0,w,t) &= 0 &\quad w\in (0,1)\\
\h(0,0,t) &= 0 &\\
\h(0,1,t)&= \ca(0,1,0)&\\
\h(0,w,0) &= \ca(0,w,0)
\end{aligned}
\right.
\end{equation}
for all $t\in (0,T)$.
\end{lem}

Assuming Lemma \ref{boundary_lem}, keeping in mind the discussion of reparametrisation at the start of this section, an application of the maximum principle to $\h(0,w,t)$ immediately gives:

\begin{cor}
\label{h_bd_cor}
In the setting of Theorem \ref{main_new_harnack_thm}, we have
$$\ca_-\leq \h(0,w,t)\leq \ca_+$$
for all $w\in (0,1)$ and all $t\in [0,T)$.
\end{cor}

\begin{proof}[Proof of Lemma \ref{boundary_lem}]
As discussed earlier, it is a property of curve shortening flows with initially radial ends within the half-lines $L_\pi$ and 
$L_{\al}$ that over compact time intervals $[0,T_0]$, (the ends of)
the ends can each continue to be written as graphs over (the ends of) $L_\pi$ and $L_{\al}$.
If we parametrise the end of $L_\pi$ or $L_{\al}$ using an arc-length parameter $x\in [0,\infty)$ and view the end of the curve as the graph of a function 
$y(x,t)$, then $y$ solves the graphical curve shortening flow equation
\beq
\label{GCSF}
y_t=\frac{y_{xx}}{1+y_x^2}\equiv (\arctan(y_x))_x.
\eeq
Lemma \ref{end_decay_lem} from the appendix then guarantees that $y$ and all its derivatives decay exponentially as $x\to\infty$, uniformly in time $t\in [0,T_0]$.

If we choose $\psi(u, 0)$ so that $\psi(u, 0) = 0$ for $u>0$ sufficiently close to $0$, 
then $\psi(u, 0) = \al$ for $u < 1$ sufficiently close to $1$.
It then follows from the exponential decay that the angle function $\psi(u,t)$ 
satisfies  $\psi(u,t)\to 0$ as $u\to 0$ and $\psi(u, t) \to \al$ as $u \to 1$, uniformly in $t\in [0,T_0]$.
This instantly gives the claimed extension and boundary behaviour for $\Psi$ in \eqref{Psi_props}.
Moreover, the claimed equation satisfied by $\Psi(0,w,t)=\psi(w,t)$ follows instantly from the equation satisfied by $\psi$ in Lemma \ref{th_psi_lem}.

Turning to $\ca$, we note that the exponential decay of $y(x)$ and its derivatives given by Lemma \ref{end_decay_lem} imply that the integrand $*(\ga_u\wedge \ga)$ of the swept area $\ca$ must decay exponentially in any arc-length parameter (or equivalently with respect to $x$) as we exit along either end.
Indeed, up to an ambient isometry, 
we could write each end of the curve as 
$\ga=(x,y(x))$, in which case $\ga_x=(1,y'(x))$ so 
$$|\ga_x\wedge\ga|=|y(x)-xy'(x)|,$$ 
which decays exponentially  as $x\to\infty$, uniformly in $t\in [0,T_0]$.

This immediately allows us to extend $\ca$ continuously to $\overline\Sigma \times [0,T)$.

We next wish to establish the equation  for $\ca(0,w,t)$ at an arbitrary time $\tau\in (0,T)$.
To do this, we take an arc-length parameter $u$ at time $t=\tau$ and  follow the proof of Lemma \ref{AH_eqs_lem} verbatim, but with $v$ replaced by $0$, computing always at time $t=\tau$. The exponential decay of $y$ and its derivatives along the ends allows us to differentiate under the integral sign in the definition of $\ca(0,w,t)$ to give
\beqa
\label{At2}
2\pl{}{t}\ca(0,w,t) &= *\int_{0}^w (\ga_{tu}\wedge\ga+\ga_u\wedge \ga_t) \, du\\
&= *\big[\ga_t\wedge\ga\big](w)+*\,2\int_{0}^w \ga_u\wedge \ga_t\, du,
\eeqa
replacing \eqref{At}.
But from \eqref{Aww} we have
$$2\ca_{ww} 
=*\big[\ga_t\wedge \ga\big](w),$$
and similarly to before we have
$$\Psi(0,w,\tau)=\int_0^w \ka(u)du=*\int_{0}^w \ga_u\wedge \ga_t\, du,$$
which implies the PDE of \eqref{A_props} at our arbitrary time $t=\tau$.

If we make the same simplification for $w$ as we did for $v$, 
replacing $w$ by $1$,
then the computation above simplifies
even more. We obtain
\beq
\pl{}{t}\ca(0,1,t) =*\int_{0}^1 \ga_u\wedge \ga_t\, du=\Psi(0,1,\tau)=\al,
\eeq
which integrates to give the final part of \eqref{A_props}.

The  properties of $\h$ claimed in \eqref{h_props} follow immediately from \eqref{Psi_props} and \eqref{A_props}.
\end{proof}

It remains to prove the main Harnack theorem, but 
before doing so we briefly comment on the symmetries of the swept area and turning function mentioned in the introduction.

\begin{rmk}\label{symmetry_rmk}
    Recall that both the swept area $\ca$ and turning function $\Psi$ of a curve $\ga : (0,1) \to \R^2$ \emph{negate} under reflections and under orientation-reversing reparametrisations.
    Specifically, 
    if $\bar\ga$ is the reflection of $\ga$ across any line through the origin (e.g.~the $x$--axis), and $\bar\ca$ and $\bar\Psi$ are the swept area and turning function of $\bar\ga$ respectively, 
    then $\bar\ca(v, w) = - \ca(v,w)$ and $\bar\Psi(v,w) = -\Psi(v,w)$,
    and if $\hat{\ga}$ is the backward parametrisation of $\ga$ given by $\hat\ga(u) = \ga(1-u)$, and $\hat\ca$ and $\hat\Psi$ are the swept area and turning function of $\hat\ga$ respectively, then $\hat\ca(v, w) = - \ca(1-w, 1-v)$ and $\hat\Psi(v, w) = - \Psi(1-w, 1-v)$.
    In particular, if we reflect \emph{and} reverse $\ga$ to get a new curve $\hat{\bar \ga} =: \widetilde\ga$, and let $\widetilde\ca$ and $\widetilde\Psi$ be the swept area and turning function of $\widetilde\ga$ respectively, then $\widetilde\ca(v, w) = \ca(1-w, 1-v)$ and $\widetilde\Psi(v,w ) = \Psi(1-w, 1-v)$.
    Geometrically, reflecting and reversing $\ga$ has the effect of interchanging the ends of $\ga$ while keeping the global properties the same.
\end{rmk}

\begin{proof}[Proof of Theorem \ref{main_new_harnack_thm}]
As discussed at the beginning of this section, we are free to reparametrise so that 
$\ci=(0,1)$ and $\ga_t = \vec{\ka}$.
This makes $\ca$, $\Psi$ and $\h$ functions of $\Si\times [0,T)$, where $\Si$ is defined in \eqref{Si_def}.
Because we are considering curve shortening flows with initially radial ends, by Lemma \ref{boundary_lem} the functions $\ca$, $\Psi$ and $\h$ extend continuously to
$\overline{\Si}\times [0,T)$. 

By Corollary \ref{h_bd_cor}, we know that 
$$\ca_-\leq \h(0,w,t)\leq \ca_+$$
for all $w \in (0,1)$ and $t \in [0,T)$, and by reflecting and reversing $\ga$ as discussed in Remark \ref{symmetry_rmk}, we can also deduce that
$$\ca_-\leq \h(v,1,t)\leq \ca_+$$
for all $v\in (0,1)$ and $t\in [0,T)$.

Trivially, we have that $\h(u,u,t)=0$ for all $u\in (0,1)$ and $t\in [0,T)$.
Combining all these facts implies that 
$$\ca_-\leq \h \leq \ca_+$$
throughout $\partial\Si\times [0,T)$, and also on $\overline{\Si}\times\{0\}$, i.e.~on the entire parabolic boundary of $\Si\times [0,T)$. 

This allows us to apply the maximum principle to the heat equation
\eqref{wedge_harnack_eq} governing the evolution of $\h$.
We deduce that \eqref{gen_Harn} holds throughout $\overline{\Si}\times [0,T)$ as required.
\end{proof}

\begin{proof}[{Proof of Theorem \ref{wedge_harnack_thm1}}]
By \eqref{ca_cv}, in the context of Theorem \ref{wedge_harnack_thm1}
we have $\cv(t)=\ca(0,\be,t)$. 
The formula \eqref{total_V} for the area then follows from \eqref{A_props} in Lemma \ref{boundary_lem} with $\al := \pi - \be$, after reparametrising as at the beginning of this section.

The main inequality \eqref{wh1} is a translation of \eqref{gen_Harn} of Theorem \ref{main_new_harnack_thm} in the special case of polar graphical flows, since $\ca_-=0$ and $\ca_+=\cv_0$.

Denote by $r_\be:(0,\beta)\to (0,\infty)$, parametrised as $r(\vph)$, the radial polar coordinate of the $\be$-wedge solution at time $t=1$.
The inequalities of \eqref{wh2pi} follow because 
$$0\leq \cv(\vph_0,\vph_1,t)\leq \cv(0,\beta,t)=\cv_0$$
in the case $\be=\pi$.
The inequalities of \eqref{wh2} follow because
$r(\vph,t)\geq r_\be(\vph,t):=t^\half r_\be(\vph)$ by the comparison principle,
and so on the one hand
$$\cv(\vph_0,\vph_1,t)- t\cv_\be(\vph_0,\vph_1)\geq 0,$$
and on the other hand
$$\cv(\vph_0,\vph_1,t)- t\cv_\be(\vph_0,\vph_1)
\leq \cv(0,\be,t)- t\cv_\be(0,\be)
=[\cv_0+(\pi-\be)t]-(\pi-\be)t=\cv_0.$$
By \eqref{wasL71} this implies that
$$0\leq \cv(\vph_0,\vph_1,t)- t\Psi_\be(\vph_0, \vph_1)\leq \cv_0,$$
which can be added to the negation of  \eqref{wh1} to give \eqref{wh2}.
\end{proof}

\section{Estimates on when solutions of curve shortening flow with radial ends become graphical}
\label{arjun_sect}

In this section we prove Theorem \ref{ca_control_thm}, and use it to deduce Theorem \ref{graphical_cor}.

Theorem \ref{ca_control_thm} will follow by applying a maximum principle argument (similar to that in the previous section), not to \eqref{wedge_harnack_eq} as in the proof of Theorem \ref{main_new_harnack_thm}, but instead to the evolution equation 
$$ \ca_t = \Delta_v \ca + \Delta_w \ca + \Psi(v, w, t)$$
 for the swept area $\ca$. 
As mentioned in the introduction, the difficulty when proving Theorem \ref{ca_control_thm}  is 
to control the forcing term $\Psi(v,w,t)$ above, which is finite but may be arbitrarily large for arbitrary $(v,w,t)$ if $\ga(\cdot,t)$ spirals many times as in Figure \ref{fern_fig}.
Consequently, a naive application of the maximum principle will not give us any usable information.

To obtain the sharp estimate for the swept area, Theorem  \ref{ca_control_thm}, which we require to deduce Theorem \ref{graphical_cor} from Theorem \ref{main_new_harnack_thm}, we need to understand the behaviour of $\Psi$ at extremal points of $\ca$.
This is recorded in the following two lemmas:

\begin{lem}\label{lem:turnangbdd}
Let $\gamma : (0,1) \times [0,T) \to \R^2$ be a curve shortening flow in the setting of Lemma \ref{boundary_lem}, but with $\al \in  [0, \pi)$, and extend the turning function $\Psi$ and the swept area $\ca$ continuously to $\overline\Sigma \times [0, T)$, as described there.
Then for any $t \in (0, T)$,
\begin{enumerate}[(i)]
\item  if $(v_*, w_*) \in \overline\Sigma$ is such that $$\ca(v_*, w_*, t) = \sup_{(v,w) \in \Sigma} \ca(v,w,t),$$ then
\beq
\Psi(v_*, w_*, t)  \leq \al . \label{turnupbdd}
\eeq
\item if $(v_*, w_*) \in \overline\Sigma$ is such that $$\ca(v_*, w_*, t) = \inf_{(v,w) \in \Sigma} \ca(v,w,t),$$ then
\beq \Psi(v_*, w_*, t) \geq 0 . \label{turnlowbdd} \eeq
\end{enumerate}
\end{lem}

\begin{lem}\label{lem:turnangbddbdry}
In the setting of Lemma \ref{lem:turnangbdd}, 
for any $t  \in (0, T)$,
\begin{enumerate}[(i)]
\item  if $w_* \in [0, 1]$ is such that $$\ca(0, w_*, t) = \sup_{w \in (0,1)} \ca(0,w,t),$$ then
\beq
\Psi(0, w_*, t)  \leq \al . \label{turnupbddtop}
\eeq
\item  if $w_* \in [0,1]$ is such that $$\ca(0, w_*, t) = \inf_{w \in (0,1)} \ca(0,w,t),$$ then
\beq \Psi(0, w_*, t) \geq 0 . \label{turnlowbddtop} \eeq
\end{enumerate}
\end{lem}

\subsection{Proofs of Theorems \ref{ca_control_thm} and \ref{graphical_cor} assuming Lemmas \ref{lem:turnangbdd} and \ref{lem:turnangbddbdry}}

Before proving Lemmas \ref{lem:turnangbdd} and \ref{lem:turnangbddbdry}, we use them
to prove  Theorems \ref{ca_control_thm} and \ref{graphical_cor}.

\begin{proof}[Proof of Theorem \ref{ca_control_thm}]
As in the proof of Theorem \ref{main_new_harnack_thm}, we may assume without loss of generality that $\ci = (0,1)$ and $\ga_t = \vec{\kappa}$, and that the turning function $\Psi$ and swept area $\ca$ extend continuously to $\overline \Sigma \times [0, T)$.
We wish to apply the maximum principle to the equation \eqref{ca_eq} for $\ca$ over $\Sigma \times [0, T)$, but to do so we first need to understand how $\ca$ behaves on $\partial \Sigma \times [0, T)$. 

Let $\ep > 0$, and define 
$$ \ca^\ep(v, w, t) := \ca(v, w, t) - \al t - \ep(1 + t). $$
Then 
\beql{Atime0}
\ca^\ep(v, w, 0) = \ca(v, w, 0) - \ep \leq \ca_+ - \ep
\eeq
and, from \eqref{A_props}, we compute that $\ca^\ep(0, w, t)$ satisfies
\begin{equation}
\label{A^ep_props}
\left\{
\begin{aligned}
\left(\pl{}{t}-\lap_w\right)\ca^\ep(0,w,t) &= 
\Psi(0,w,t) - \al - \ep &\quad w\in (0,1)\\
\ca^\ep(0,0,t) &= -\al t - \ep(1 + t) \leq -\ep \leq \ca_+ - \ep &\\
\ca^\ep(0,1,t)&= \ca(0,1,0) - \ep(1+t) \leq \ca_+ - \ep&
\end{aligned}
\right.
\end{equation}
for all  $t \in (0, T)$.
Suppose for a contradiction that there exists a first time  $t_* \in [0, T)$ for which 
\beql{t*timemax} 
\sup_{w \in (0,1)} \ca^\ep(0, w, t_*) \geq \ca_+, 
\eeq
and let $w_* \in [0,1]$ be such that
\begin{equation*} 
 \ca^\ep(0, w_*, t_*) = \sup_{w \in (0,1)} \ca^\ep(0, w, t_*).
 \end{equation*}
Then \eqref{Atime0} and \eqref{A^ep_props} guarantee that, respectively, $t_* > 0$ and $0 < w_* < 1$. 
As $\ca(0, w, t_*)$ and $\ca^\ep(0, w, t_*)$ differ by a constant which is independent of $w$, we have that $w_*$ is also a maximum for $w \mapsto \ca(0, w, t_*)$. 
Using (i) from Lemma \ref{lem:turnangbddbdry} in \eqref{A^ep_props}, we deduce that
$$
\left(\pl{}{t}-\lap_w\right)\ca^\ep(0,w_*,t_*)\leq  \al - \al - \ep = - \ep < 0.
$$
Conversely, since $t_* > 0$ is the first time at which \eqref{t*timemax} holds, and since $w_*$ is an interior maximum for $w \mapsto \ca^\ep(0, w, t_*)$, we also deduce that
$$
\left(\pl{}{t}-\lap_w\right)\ca^\ep(0,w_*,t_*) \geq  0. 
$$
These two inequalities are clearly contradictory.
Hence \eqref{t*timemax} could not have held, and in fact
$$\ca^\ep (0, w, t) \leq \ca_+$$
for all $w \in (0, 1)$ and all $t \in[0, T)$.
Taking $\ep \downto 0$ and rearranging gives
\beql{bdry_upbdd} \ca(0, w, t) \leq \al t + \ca_+ . \eeq
%
Similarly, 
\beql{bdry_lowbdd} \ca(0, w, t) \geq \ca_- \eeq
for all $w \in (0,1)$ and all $t \in [0, T)$, this time using (ii) from Lemma \ref{lem:turnangbddbdry}.

By reflecting and reversing $\ga$ (cf.~Remark \ref{symmetry_rmk}), 
we deduce that
\beql{bdry_v1bdd} \ca_- \leq \ca(v, 1, t) \leq \al t + \ca_+  \eeq
for all $v \in (0,1)$ and all $t \in [0, T)$.

Taking the bounds  \eqref{bdry_upbdd}, \eqref{bdry_lowbdd} and \eqref{bdry_v1bdd}  together with \eqref{harn_bds} and the trivial observation that $\ca(u,u,t) = 0$ for all $u \in [0,1]$ gives that 
\beql{Abdrycntrl}
\ca_- \leq \ca(v, w, t) \leq \al t + \ca_+ 
\eeq
for all $(v, w, t)$ with $(v, w) \in \partial \Sigma$ or $t = 0$.

We can now apply the same maximum principle argument to \eqref{ca_eq} to deduce control on all of $\Si \times [0, T)$.
For an arbitrary $\ep > 0$, we have from
 \eqref{ca_eq} that
$$ \ca^\ep_t - \Delta_v \ca^\ep - \Delta_w \ca^\ep = \Psi(v, w, t) - \al - \ep$$
for all $(v, w) \in \Sigma$ and $t \in (0, T)$,
and from \eqref{Abdrycntrl} that
$$ \ca^\ep(v, w, t) \leq \ca_+ - \ep(1+t) \leq \ca_+ - \ep$$
whenever $(v, w, t)$ lies on the parabolic boundary of $\Si \times [0, T)$.

Suppose there exists a first time $t_* \in (0, T)$ for which $\ca^\ep(\cdot, \cdot, t_*) \geq \ca_+$, and let $(v_*, w_*) \in \Sigma$ be a maximum for $\ca^\ep(\cdot, \cdot, t_*)$. 
As before, $(v_*, w_*)$ is also a maximum for $\ca(\cdot, \cdot, t_*)$, and hence by (i) of Lemma \ref{lem:turnangbdd} we have
$$ \Psi(v_*, w_*, t_*) \leq \al .$$
Therefore
$$ \ca^\ep_t - \Delta_v \ca^\ep - \Delta_w \ca^\ep \leq - \ep  $$
at $(v_*, w_*, t_*)$. 
However, we also have 
$$ \ca^\ep_t - \Delta_v \ca^\ep - \Delta_w \ca^\ep \geq 0$$
by virtue of $(v_*,w_*,t_*)$ being a first space-time maximum for $\ca^\ep$.
These two inequalities are clearly contradictory, so our supposition could not have held, i.e.~we must have had
$$\ca^\ep(v, w, t) \leq \ca_+$$
for all $(v, w, t ) \in \Sigma \times [0 , T)$.
Taking $\ep \downto 0$ and rearranging gives the second inequality of \eqref{ca_control_thm_equation}. 
The argument for the first inequality of \eqref{ca_control_thm_equation} is similar, but uses (ii) of Lemma \ref{lem:turnangbdd} in place of (i).
\end{proof}

\begin{proof}[Proof of Theorem \ref{graphical_cor}]
As before, we may assume without loss of generality that $\ci = (0,1)$  and $\ga_t = \vec{\ka}$, and by Lemma \ref{boundary_lem}, that  $\Psi$, $\ca$ and $\h$ extend continuously to $\overline \Sigma \times [0, T)$. 
In particular, we may take $\psi(0, t) = - \frac\pi2 + \frac\be2$ and $\psi(1,t) = \frac\pi2 - \frac\be2$.

From Theorem \ref{main_new_harnack_thm} we have that 
$$
\ca_- \leq \h(v,w,t) = \ca(v,w,t) - t\Psi(v,w,t) \leq \ca_+,
$$
which rearranges to 
$$
 - \frac{\ca_+ - \ca(v,w,t)}{t} \leq \Psi(v,w,t) \leq \frac{\ca(v,w,t) - \ca_-}{t}
$$
for all $(v,w) \in \Sigma$ and all $t \in (0, T)$.
Applying Theorem \ref{ca_control_thm} with $\al : = \pi - \be \geq 0$, we deduce that
\beql{Psibdd2}
- \frac{\ca_+ - \ca_-}{t} \leq \Psi(v,w,t) \leq 
[\pi - \be] + \frac{\ca_+ - \ca_-}t
\eeq
 for all $(v, w) \in \overline{\Sigma}$.
Applying \eqref{Psibdd2} with $v = 0$ and $w = u$, keeping in mind $\psi(0,t) = -\frac\pi2 + \frac\be2$, gives after rearranging that
\beql{Tauthres}
-\psi_\text{max}(t): = -\left[\frac\pi2 - \frac\be2 + \frac{\ca_+ - \ca_-}{t} \right] \leq \psi(u,t) \leq \frac\pi2-\frac\be2 + \frac{\ca_+ - \ca_-}{t} = \psi_\text{max}(t)
\eeq
for all $u \in (0,1)$ and $t \in (0, T)$, as illustrated by the solid red arc in Figure \ref{fig:angleconstraint}.

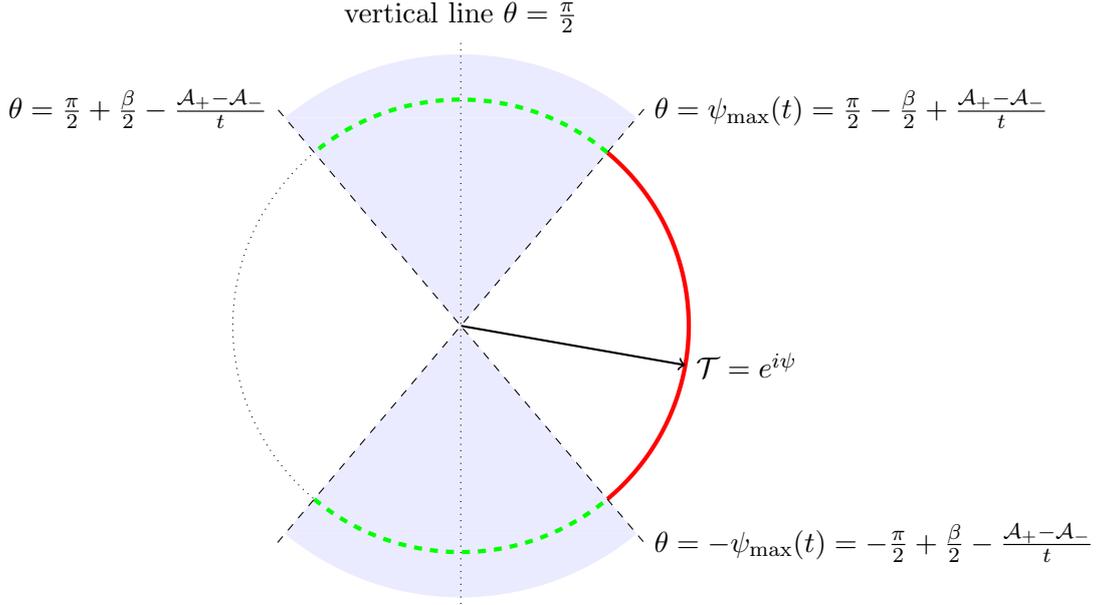
\begin{figure}
\centering
\begin{tikzpicture}[scale=1.5, rotate=-30]
\draw[dotted, samples=360,domain=0:360] plot ({2*cos(\x)}, {2*sin(\x)}); 
\draw[ultra thick, red, samples=100,domain=-20:80] plot ({2*cos(\x)}, {2*sin(\x)}); 
\draw[thick, ->] (0,0) to (2*1*.9397, 2*1*0.3420) node[right]{$\Tau = e^{i\psi}$};
\fill[blue!8, samples=80,domain=80:160] plot ({2.4*cos(\x)}, {2.4*sin(\x)}); 
\fill[blue!8] (2.4*0.1736, {2.4*0.9848}) to (2.4*-0.9397, 2.4*0.3420) to (0,0) ; 
\fill[blue!8, samples=80,domain=-100:-20] plot ({2.4*cos(\x)}, {2.4*sin(\x)}); 
\fill[blue!8] (2.4*-0.1736, {2.4*-0.9848}) to (2.4*0.9397, 2.4*-0.3420) to (0,0) ; 
\draw[ultra thick, dashed, green, samples=80,domain=80:160] plot ({2*cos(\x)}, {2*sin(\x)}); 
\draw[dashed] (2.5*0.1736, {2.5*0.9848}) node[right]{$\th = \psi_\text{max}(t) = \tfrac\pi2 - \tfrac\be2 + \tfrac{\ca_+-\ca_-}t $} to (0,0) to (2.5*-0.9397, 2.5*0.3420) node[left]{$\textstyle \th = \frac\pi2 + \frac\be2 - \frac{\ca_+ - \ca_-}{t}$}; 
\draw[ultra thick, dashed, green, samples=80,domain=-100:-20] plot ({2*cos(\x)}, {2*sin(\x)}); 
\draw[dashed] (2.5*-0.1736, {2.5*-0.9848}) to (0,0) to (2.5*0.9397, 2.5*-0.3420)  node[right]{$\th = -\psi_\text{max}(t) = -\tfrac\pi2 + \tfrac\be2 - \tfrac{\ca_+-\ca_-}t $}; 
\draw[dotted] (2.5*-.5, 2.5*0.8660) node[above]{vertical line $\th =\tfrac\pi2$} to (-2.5*-.5, -2.5*0.8660) ;
\end{tikzpicture}
\caption{The solid red arc illustrates the values that $\Tau$ can take in $S^1 \subset \C$ at a time $t > \frac{2(\ca_+ - \ca_-)}\be$, and consequently the shaded blue double sector  illustrates the region in which $\ga(\cdot, t)$ is polar graphical. 
}
\label{fig:angleconstraint}
\end{figure}

Once $t > \frac{2(\ca_+ - \ca_-)}\be$, that is, once $\frac{\ca_+ - \ca_-}{t} < \frac\be2$, we have that 
$$ \psi_\text{max}(t) =  \frac\pi2 - \frac\be2 + \frac{\ca_+-\ca_-}{t} < \frac\pi2.$$
Therefore $$ -\frac\pi2 <  \psi(u, t) < \frac\pi2$$ for all $u \in (0,1)$, implying that  $\ga(\cdot, t)$ must cross any vertical line (dotted in Figure \ref{fig:angleconstraint}) transversally and precisely once, and consequently that we can write $\ga(\cdot, t)$ as a graph of a function $y(\cdot, t)$ over the $x$--axis.
Moreover, we have the relationship 
$$ \arctan(y_x(x, t)) = \psi(u, t), $$
where $\ga(u, t) = (x, y(x, t))$, which when combined with \eqref{Tauthres} gives \eqref{graphbdd}.

Now fix $t > \frac{2(\ca_+ - \ca_-)}\be$. 
Intuitively, the reason for the polar graphicality is the same as the reason for the graphicality over the $x$--axis: \eqref{Tauthres} is saying, loosely, that $\Tau$ is `looking horizontal', and so if we have a line which `looks vertical' and is crossed by $\ga(\cdot, t)$, then the crossing must be transversal.

For graphicality over the $x$--axis, we needed to check whether $\ga(\cdot,t )$ crossed the vertical line $x + \{ re^{\frac{i\pi}2} : r \in \R \}$ for any $x \in \R$ transversely. 
For polar graphicality, we need to check that $\ga(\cdot , t)$ crosses $\{ re^{i\th} : r \in \R \}$ transversely as $\theta$ ranges over some interval.
We claim that the maximal interval is
$$ \s =\left( \frac\pi2 - \frac\be2 + \frac{\ca_+-\ca_-}{t}, \frac\pi2 + \frac\be2 -\frac{\ca_+-\ca_-}{t}\right).$$
This can be easily seen from inspecting Figure \ref{fig:angleconstraint}, as 
$$ \left\{ r e^{i\th} : r \in \R, \, \th \in \s \right\} $$ 
is exactly the maximal double sector (illustrated as the shaded blue cone over the pair of dashed green arcs) in which $\Tau$ cannot lie.
\end{proof}

\subsection{Ideas and ingredients for the proofs of Lemmas \ref{lem:turnangbdd} and \ref{lem:turnangbddbdry}}\label{sec:ideasingredients}

On a curve $\ga$ in $\R^2$ we will need to keep track of the angle between the tangent $\Tau(v)$ and the chord connecting 
$\ga(v)$ with other points $\ga(u)$.
\begin{defn}\label{def:incfunc}
Let $\ga : \ci \to \C \simeq \R^2$ be an embedded smooth curve, where $\ci \subseteq \R$ is an interval, and let $v \in \ci$.
We define the \emph{winding function $\th_v : \ci \to \R$ of $\ga$ (based at $v$)} to be the continuous function given by
\beql{incfuncform}
e^{i\th_v(u)} \Tau(v) = 
\begin{cases} \displaystyle
\frac{\ga(u) - \ga(v)}{|\ga(u) - \ga(v)|} &\text{for } \, u > v \\
\displaystyle-\frac{\ga(u) - \ga(v)}{|\ga(u) - \ga(v)|} &\text{for } \, u < v 
\end{cases}
\eeq
with $\th_v(v) =0 $, where we view the tangent vector $\Tau(v) \in S^1 \subset \C$. 
\end{defn}

Our winding functions negate under reflections, and are \emph{invariant} under orientation-reversing reparametrisations such as the one given in Remark \ref{symmetry_rmk}.

According to Lemmas \ref{lem:turnangbdd} and \ref{lem:turnangbddbdry}, 
our objective is to control the turning function $\Psi$, and that will be achieved via the following lemma.

\begin{lem}\label{lem:dC}
Let $\gamma : [c,d] \to \R^2$ be an embedded smooth arc. 
Let $\Psi$ be the turning function of $\gamma$, cf.~Definition \ref{Psi_def}, and let $\th_c$ and $\th_d$ be winding functions of $\ga$ based at $c$ and $d$ respectively, cf.~Definition \ref{def:incfunc}.  
Then
\beql{dC} \Psi(c,d) = \th_c(d) - \th_d(c). \eeq
\end{lem}

Recall the classical result in planar geometry asserting that a positively oriented simple smooth  planar curve has turning number precisely $+1$, which was used in Lemma \ref{boundary_lem}.
The proof of Lemma \ref{lem:dC} is a transcription of part of the proof of the classical result; see Theorem 2 of \S5-7 in \cite{dC}.

\begin{proof}[Proof of Lemma \ref{lem:dC}]
Without loss of generality, we may take $c = 0$ and $d = 1$.
Define the continuous map 
$\sigma : \overline\Sigma \to S^1\embed\R^2$ by 
\beql{}
\sigma(v, w) = \begin{cases} \displaystyle \frac{\ga(w) - \ga(v)}{|\ga(w) - \ga(v)|} &\text{if }\, v < w \\ \displaystyle \frac{\ga_u(v)}{|\ga_u(v)|} &\text{if } \, v = w. \end{cases}
\eeq
Consider the diagonal and edge paths $D, E : [0,1] \to \overline\Sigma$ given by
$$D(u) = (u,u) \qquad \text{and} \qquad E(u) = \begin{cases} (0, 2u) &\text{for } 0 \leq u \leq \frac12 \\ (2u-1,1) &\text{for } \frac12 \leq u \leq 1. \end{cases}$$
The two paths are homotopic in $\overline\Sigma$, and thus so are their push-forwards $\sigma_* D = \si \circ D$ and $\sigma_*E = \si \circ E$ in $S^1$. 
By definition, the left-hand side of \eqref{dC} is the angle turned by $\sigma_*D$ and the right-hand side is the angle turned by $\sigma_*E$; hence they coincide.
\end{proof}

As well as compact arcs, the proofs of Lemmas \ref{lem:turnangbdd} and \ref{lem:turnangbddbdry} will also require us to consider non-compact arcs with one end asymptotic to a line. 
We collect these applications of Lemma \ref{lem:dC} below. 

\begin{cor}\label{cor:dC}\begin{enumerate}[(i)]
\item If $\gamma : (0, d] \to \C \simeq \R^2$ is an embedded smooth  curve for which $\ga(u)$ is $C^1$--asymptotic to a line as $u\downto0$, then
 \begin{equation}\label{cDends}
 \Psi(0, d) := \lim_{c \downarrow 0} \Psi(c, d) =  -\lim_{c \downto 0}{ \th_d(c) }.
 \end{equation}
\item If $\gamma : [c, 1) \to \C \simeq \R^2$ is an embedded smooth  curve for which $\ga(u)$ is $C^1$--asymptotic to a line as $u \uparrow 1$, then 
\begin{equation}\label{cDendf}
 \Psi(c, 1)  := \lim_{d \upto 1} \Psi(c, d) =   \lim_{d \upto 1} {\th_c(d)}.
 \end{equation}
 \end{enumerate}
\end{cor}

\begin{proof}
By considering a reverse parametrisation, it suffices to prove only (i).
We will show that the first term on the right-hand side of \eqref{dC}  will vanish in the limit as $c \downto 0$, i.e.~that
\beql{limeq}
\lim_{c \downto 0} { \th_c(d) } = 0.
\eeq
Let $\ep \in (0,\frac1{100})$ be arbitrary. 
By rotating $\ga$, it suffices to assume that $\ga(u)$ asymptotes to the negative $x$-axis and $\psi(u) \to 0$ as $u \downto 0$, where $\Tau(u) = e^{i\psi(u)} $.
By scaling, it suffices to assume that $\ga$ lies within the strip $\{x + iy : |y| < 1\} \subset \C$.
And by translating, it suffices to assume that $\ga(d)$ lies to the right of the $y$--axis and $|\psi(u)| \leq \frac\ep{100}$ whenever $\ga(u)$ lies to the left of the $y$--axis.  

In this setup, for any point $\ga(c)$ with sufficiently large and negative $x$--coordinate ($\ga(c) = x + iy$ with $x < -\frac{100}\ep$ will do), the cone of opening angle $\ep$ around $\Tau(c)$ contains the whole half-strip $\{ x + iy : x \geq 0, \, |y| < 1 \}$, as well as the arc of $\ga$ joining $\ga(c)$ to the $y$--axis. 
Consequently $|\th_c(d)| < \ep$, and as $\ep  \in (0,\frac1{100})$ was arbitrarily small, we  conclude that \eqref{limeq} holds. 
See Figure \ref{fig:dConeend} for an illustration.
\end{proof}

\begin{figure}
\centering
\begin{tikzpicture}[scale=1]

\draw[fill, blue!4] (2, 1.5) to (-9,.2) to (2, -2) ;
\draw[dotted] (2, 1.5) to (-9,.2) to (2, -2)             ;

\draw[dashed, ->] (-10,0) to (2.5, 0) node[right]{$x$} ;
\draw[dotted] (-10, 1) to (2.5, 1) node[right]{$y=1$};
\draw[dotted] (-10, -1) to (2.5, -1) node[right]{$y=-1$};
\draw[dashed, ->] (0, -2) to (0, 1.5) node[above]{$y$};

\draw[very thick, ->] (-9, .2) to (-8.5, .18) ;
\draw[thick, ->] (-9, .2) to[out=-1, in=190]  (0, .1) to[out=10, in=180] (1, -.9) to[out=0, in=-45] (1.5, -.5) to[out=180-45, in=-135] (1,.5) to[out=45, in = 150] (1.5, .5) to[out=-30, in=150] (1.7, 0.4) ;
\draw[very thick, dotted] (-9, .2) to[out=179, in=1] (-10, .1)  ;
\node at (1.5,.5) [circle,fill,inner sep=1pt]{} ;
\draw (1.5, .5) node[above]{$\ga(d)$} ;
\node at (-9,.2) [circle,fill,inner sep=1pt]{} ;
\draw (-9, .2) node[above]{$\ga(c)$} ;
\draw[->] (-8.75, -1.1) node[below]{$\Tau(c)$} to (-8.75, .19) ;

\draw[->] (-5, 1.1) node[above]{cone of opening angle $\ep$ around $\Tau(c)$} to (-5, 0.5);

\end{tikzpicture}
\caption{Illustration to show that if the arc $\ga: [c, d]$ is contained within the cone of opening angle $\ep$ around the tangent vector $\Tau(c)$, then necessarily it must hold that $|\th_c(d) | < \ep$.} 
\label{fig:dConeend}
\end{figure}
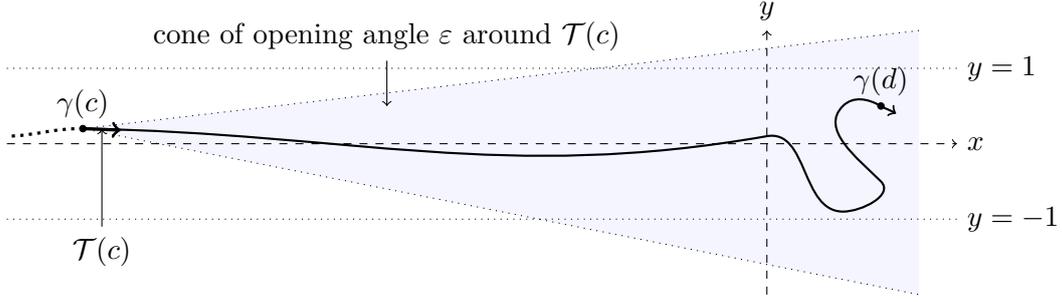

Given Lemma \ref{lem:dC} and our objective from Lemmas \ref{lem:turnangbdd} and \ref{lem:turnangbddbdry}, 
we see that we must control the winding angle $\th_c(d)$ when $(c,d)$ represent an extremal of $\ca$.
We will address that issue in the next section.

\subsection{Proofs of Lemmas \ref{lem:turnangbdd} and \ref{lem:turnangbddbdry}}

We start by stating a lemma of a similar character to Lemmas \ref{lem:turnangbdd} and \ref{lem:turnangbddbdry}. 
Although the following lemma holds for the trivial static flow, we omit that case for ease of presentation.

\begin{lem}\label{lem:turnangbddred}
Let $\gamma : (0,1) \times [0,T) \to \R^2$ be a curve shortening flow in the setting of Lemma \ref{boundary_lem} that is \emph{not} the static flow of the $x$--axis, in particular with $\al \in (-\pi, \pi)$ possibly negative, and extend the turning function $\Psi$ and the swept area $\ca$ continuously to $\overline\Sigma \times [0, T)$, as described there.
Suppose for some $t \in (0, T)$ that $(v_*, w_*) \in \overline\Sigma$ is such that either $\ca(v_*, w_*, t) = \sup_{(v, w) \in \Sigma} \ca(v, w, t)$,  or
$v_* = 0$ and $\ca(0, w_*, t) = \sup_{w \in (0,1)} \ca(0, w, t)$.
Then 
\beql{turnangbddred}
\Psi(v_*, w_*, t) \leq \max\{ 0, \al \}.
\eeq
\end{lem}

The significance of Lemma \ref{lem:turnangbddred} is that it implies both Lemmas \ref{lem:turnangbdd} and \ref{lem:turnangbddbdry}.

\begin{proof}[Proof that Lemma \ref{lem:turnangbddred} implies Lemmas \ref{lem:turnangbdd} and \ref{lem:turnangbddbdry}]
Lemmas \ref{lem:turnangbdd} and \ref{lem:turnangbddbdry} are both trivial when $\ga$ is the static flow of the $x$--axis.
Moreover, for the $\alpha \in [0, \pi)$ from either Lemma \ref{lem:turnangbdd} or \ref{lem:turnangbddbdry}, which satisfies $\max\{0,\al\} = \al $, we immediately see that Lemma \ref{lem:turnangbddred} implies case (i) of either Lemma \ref{lem:turnangbdd} or \ref{lem:turnangbddbdry} respectively.

Fix $ t \in (0, T)$, and consider $(v_*, w_*) \in \overline\Sigma$ as in case (ii) of either Lemma \ref{lem:turnangbdd} 
 or \ref{lem:turnangbddbdry}. 
Let $\bar\ga$ be the reflection of $\ga$ in the $x$--axis, and let $\bar\ca$ and $\bar\Psi$ be the swept area and turning function of $\bar\ga$ respectively.
By Remark \ref{symmetry_rmk}, we have that $\bar\ca(v, w,t) = - \ca(v, w,t)$ and $\bar\Psi(v, w, t) = -\Psi(v,w,t)$.
Therefore $\bar\ga$ is a curve shortening flow with initially radial ends at angles $\pi$ and $\bar\al := -\alpha \leq 0$, 
and, respectively, either $( v_*, w_*)$ \emph{maximises} $\bar \ca(\cdot, \cdot, t)$ or $w_*$ \emph{maximises} $\bar \ca(0, \cdot, t)$.
Applying Lemma \ref{lem:turnangbddred} with $\bar\al$ replacing $\al$, $\bar \ca$ replacing $\ca$ and $\bar \Psi$ replacing $\Psi$, we deduce that $\bar \Psi( v_*, w_*) \leq \max\{0, \bar\al \} = 0 $. 
This implies, respectively, either \eqref{turnlowbdd} or \eqref{turnlowbddtop}.
\end{proof}

Our overarching goal for this section will therefore be achieved once we prove Lemma \ref{lem:turnangbddred}. 

Fix $t \in (0, T)$, and suppress the dependence of $\ga$, $\ca$ and $\Psi$ on $t$ here and below.
The only information we need from viewing $\ga = \ga(\cdot, t)$ as coming from  curve shortening flow, rather than being an arbitrary embedded smooth curve whose ends are strongly asymptotic to the ends of radial half-lines, comes from Angenent's intersection principle \cite{angenent1991parabolic}. 
Specifically, we will use that  $\ga$ crosses lines and line segments discretely.

Also fix  $(v_*, w_*) \in \overline\Sigma$ maximising $\ca = \ca(\cdot, \cdot, t)$ in either of the senses in Lemma \ref{lem:turnangbddred}.

It will be convenient for the proof of Lemma \ref{lem:turnangbddred} to view $\ga = \ga(\cdot, t)$ in a different system of coordinates.
Write $\ga(u) =: r(u) e^{i (\pi - \varphi(u))}$, so that (given enough regularity on $r$ and $\vph$)
\beql{xietaint} \ca(v, w) =  \frac12 \int_{v}^{w} r(u)^2 \varphi_u \, du =: \int_{v}^{w} \rho(u) \varphi_u \, du, \eeq
where $ \rho(u) := \frac{r(u)^2}{2} \geq 0 $ and $\varphi(u) \in \R$. 
If $\ga$ misses the origin, then this can be achieved in the usual fashion by taking $r(u) = |\ga(u)|$ and selecting a continuous choice of $\vph$; 
if $\ga(u_0) = 0$ for some 
$u_0 \in (0, 1)$, then we must allow $r$ to change sign at $u_0$ to select a continuous $\vph$.
Either construction leads to smooth $\rho$ and $\vph$.
Note that we are abusing notation by thinking of $\varphi$ both as the polar angle $\pi - \th$, where $\th$ is the usual polar angle, and as a function of $u$.

The following lemma gives basic control on the swept areas at maximal points. 

\begin{claim}\label{cl:sweptarea}
\begin{enumerate}[(1)]
    \item If $v_* > 0$ and
    \begin{enumerate}[(i)]
        \item $0 \leq v_1 \leq v_*$, then \beq \ca(v_1, v_*)  \leq  0 \label{v1wrongsign}. \eeq
        \item $v_* \leq v_2 \leq w_*$, then \beq  \ca(v_*, v_2) \geq 0. \label{v2wrongsign} \eeq
    \end{enumerate}
    \item If $w_* < 1$ and 
    \begin{enumerate}[(i)]
        \item $v_* \leq w_1 \leq w_*$, then \beq \ca(w_1, w_*) \geq 0. \label{w1wrongsign} \eeq
        \item $ w_* \leq w_2 \leq 1$, then \beq \ca(w_*, w_2) \leq 0. \label{w2wrongsign} \eeq
      \end{enumerate}
\end{enumerate}
\end{claim}

The claim is an immediate consequence of additivity and maximality. 
Nevertheless, we include a proof for completeness, and to highlight how the natural symmetries introduced in Remark \ref{symmetry_rmk} can and will be used to streamline arguments in the rest of this section.   

\begin{proof}[Proof of Claim \ref{cl:sweptarea}]
We start with (1). 
Since the swept area is additive and $v_*$ maximises $\ca(\cdot, w_*)$, 
we have in the setting of (1i), i.e.~for $0 \leq v_1 \leq v_*$, that
$$ \ca(v_1, v_*) + \ca(v_*, w_*) =
\ca(v_1, w_*)
\leq \ca(v_*, w_*), $$
which implies,  after rearranging, \eqref{v1wrongsign}.
In the setting of (1ii), \eqref{v2wrongsign} follows from the calculation
$$\quad \ca(v_2, w_*) \leq \ca(v_*, w_*) = \ca(v_*, v_2) + \ca(v_2, w_*). $$
We can show (2) by an essentially identical argument, however, it is more instructive here, and below, to show that (1) implies (2) by the symmetries described in Remark \ref{symmetry_rmk}. (The two parts are in fact equivalent.)
Let $\widetilde \ga = \hat{\bar \ga}$ be a reflected and reversed version of $\ga$, with swept area $\widetilde\ca$ and turning function $\widetilde\Psi$ so that $\widetilde\ca(v,w) = \ca(1-w, 1-v)$ and $\widetilde\Psi(v,w) = \Psi(1-w,1-v)$.  
If we set $\widetilde v_* = 1 - w_*$ and $\widetilde w_* = 1- v_*$, then $\widetilde\ca(\widetilde v_*, \widetilde w_*) = \ca(v_*, w_*)$ is the maximum of $\widetilde\ca$. 
The inequalities $v_* \leq w_1 \leq w_* \leq w_2 \leq 1$ transform to $0 \leq \widetilde v_1 \leq \widetilde v_* \leq \widetilde v_2 \leq \widetilde w_*$, where $\widetilde v_1 = 1- w_2$ and $\widetilde v_2 = 1- w_1$, meaning that the setting of (2i) transforms to the setting of (1ii), and of (2ii) to (1i). 
Hence \eqref{v2wrongsign} in the form 
$\widetilde\ca(\widetilde v_*, \widetilde v_2)  \geq  0$
implies \eqref{w1wrongsign}, and \eqref{v1wrongsign} implies \eqref{w2wrongsign}. 
(The steps above are reversible, so the same argument shows that (2) implies (1).) 
\end{proof}

This elementary control on swept areas will ultimately be developed into global control on the winding functions. First, we use it to establish local control on the winding functions. 

One could view the claim below as a refined version of the second derivative test applied to $\ca$.
Indeed, having $(v_*, w_*)$ maximal for $\ca$, with $w_*  < 1$ and $\ga(w_*) \neq 0$, will tell us that $\ga$ is moving radially (outward or inward) at $\ga(w_*)$, and that $\vph$ has a strict local maximum at $w_*$. 
In the case $\vph_{uu}(w_*) \neq 0$, this exactly follows from the second derivative test applied to $\ca$. 
However, in general we need to consider the degenerate case of $\vph_{uu}(w_*) = 0$, as well as the case $\ga$ passes through the origin at $w_*$. 
To treat these exceptional cases,  we work via our knowledge of areas from Claim \ref{cl:sweptarea}, and use the information there to tell us  which side of the tangent line at $\ga(w_*)$ we are on near $w_*$. 
In turn, this will tell us the sign of the winding angle $\th_{w_*}$ as we enter/exit $w_*$, provided we know whether $\ga$ is entering/exiting $\ga(w_*)$ in an outward/inward direction.

Part (2) of Claim \ref{cl:windingfuncsversion2} below compiles all this information into four possible cases, split into entering/exiting $\ga(w_*)$, parts (i)/(ii), in an outward/inward direction, parts (a)/(b). 
For ease of application, we also give in (1) the four cases for entering/exiting $\ga(v_*)$ in an outward/inward direction, even though all of (1) follows from (2) by a symmetry argument.

\begin{claim}\label{cl:windingfuncsversion2}
    \begin{enumerate}[(1)]
        \item If $v_* >0$, then the tangent line to $\ga(v_*)$ passes through the origin, and 
        \begin{enumerate}[(a)]
            \item \begin{enumerate}[(i)]
                \item if $\Tau(v_1)$ points outward\footnote{by which we mean that $\Tau(v_1)$ and $\ga(v_1)$, thought of as vectors based at the origin, make an acute angle at the origin} (from the origin) for all $v_1 < v_*$ sufficiently close, then
                    \beq \th_{v_*}(v_1) > 0 \quad \text{for all $v_1 < v_*$ sufficiently close}. \label{voutback} \eeq
                \item if $\Tau(v_2)$ points outward for all $v_2 > v_*$ sufficiently close, then
                    \beq \th_{v_*}(v_2) <0 \quad \text{for all $v_2 > v_*$ sufficiently close}. \label{voutfor} \eeq
            \end{enumerate}
            
             \item \begin{enumerate}[(i)]
                \item if $\Tau(v_1)$ points inward\footnote{by which we mean $-\Tau(v_1)$ points outward} (to the origin) for all $v_1 < v_*$ sufficiently close, then
                    \beq  \th_{v_*}(v_1) < 0 \quad \text{for all $v_1 < v_*$ sufficiently close}. \label{vinback} \eeq
                \item if $\Tau(v_2)$ points inward for all $v_2 > v_*$ sufficiently close, then
                    \beq  \th_{v_*}(v_2) >0 \quad \text{for all $v_2 > v_*$ sufficiently close}. \label{vinfor} \eeq
            \end{enumerate}
        \end{enumerate}
        
        \item If $w_* <1$, then the tangent line to $\ga(w_*)$ passes through the origin, and 
        \begin{enumerate}[(a)]
            \item \begin{enumerate}[(i)]
                \item if $\Tau(w_1)$ points outward for all $w_1 < w_*$ sufficiently close, then
                    \beq \th_{w_*}(w_1) < 0 \quad \text{for all $w_1 < w_*$ sufficiently close}. \label{woutback} \eeq
                \item if $\Tau(w_2)$ points outward for all $w_2 > w_*$ sufficiently close, then
                    \beq \th_{w_*}(w_2) >0 \quad \text{for all $w_2 > w_*$ sufficiently close}. \label{woutfor} \eeq
            \end{enumerate}
            
             \item \begin{enumerate}[(i)]
                \item if $\Tau(w_1)$ points inward for all $w_1 < w_*$ sufficiently close, then
                    \beq \th_{w_*}(w_1) > 0 \quad \text{for all $w_1 < w_*$ sufficiently close}. \label{winback} \eeq
                \item if $\Tau(w_2)$ points inward for all $w_2 > w_*$ sufficiently close, then
                    \beq \th_{w_*}(w_2) <0 \quad \text{for all $w_2 > w_*$ sufficiently close}. \label{winfor} \eeq
            \end{enumerate}
        \end{enumerate}
    \end{enumerate}
\end{claim}

We will need all parts of Claim \ref{cl:windingfuncsversion2} in the proof of Lemma \ref{lem:turnangbddred} except for (1ai) and (2bii), but we include them here for completeness.

    If $\ga(v_*) = 0$, then necessarily $\Tau(v_1)$ points inward for all $v_1 < v_*$ sufficiently close and $\Tau(v_2)$ points outward for all $v_2 > v_*$ sufficiently close, i.e.~both the settings of (1bi) and (1aii) from Claim \ref{cl:windingfuncsversion2}, or from Lemma \ref{lem:windingconstraints} below, hold.
    If $\ga(v_*) \neq 0$, then $\Tau(v_*)$ either points outward or points inward, using
    that the tangent line to $\ga(v_*)$ passes through the origin, i.e.~the first assertion of (1) in Claim \ref{cl:windingfuncsversion2}.
        Therefore, by continuity, either $\Tau(u)$ points outward for all $u$ sufficiently close to $v_*$, meaning that the settings of both (1ai) and (1aii) hold, or $\Tau(u)$ points inward for all $u$ sufficiently close to $v_*$, meaning that the setting of both (1bi) and (1bii) hold. 
    Similarly, if $\ga(w_*) = 0$, then both the settings of (2aii) and (2bi) hold, and if $\ga(w_*) \neq 0$, then either both the settings of (2ai) and (2aii) hold, or both the settings of (2bi) and (2bii) hold.
    As a consequence, we have established:

\begin{claim}
\label{rmk:windingfunc}
\begin{enumerate}[(1)]
   \item  If $v_* > 0$ and we are \emph{not} in the setting of (1aii) of Claim \ref{cl:windingfuncsversion2}, then $\ga(v_*) \neq 0$, 
    $\Tau(v_*)$ points inward towards the origin, and both settings (1bi) and (1bii) hold.
    \item If $w_* < 1$ and we are not in the setting of (2bi), then 
    $\ga(w_*) \neq 0$, $\Tau(w_*)$ points outward away from the origin, and both settings (2ai) and (2aii) hold.
\end{enumerate}
\end{claim}

\begin{proof}[Proof of Claim \ref{cl:windingfuncsversion2}]
We start with (2).
If $\ga(w_*) = 0$, then it is obvious the tangent line to $\ga(w_*)$ passes through the origin, and that 
we are in the settings of both (2aii) and (2bi). 
If $\ga(w_*) \neq 0$, 
then taking the derivative of \eqref{xietaint} with respect to $w$, keeping in mind that $ w_*$ maximises $\ca(v_*, \cdot)$, gives
$$ 0 = \ca_w(v_*, w_*) = \rho(w_*) \vph_u(w_*).$$
As $\rho(w_*) > 0$, we have that $\vph_u(w_*) = 0$, and hence that
$$ \ga_u = r_u e^{i(\pi - \vph)} - i \vph_u  \cdot r e^{i(\pi - \vph)} = \frac{r_u}{r} \cdot \ga$$
at $w_*$.
This is enough to conclude that the  tangent line to $\ga$ at $\ga(w_*)$ passes through the origin also in the case $\ga(w_*) \neq 0$, which is the first assertion of (2).

We prove only (2bi); the proofs of (2ai), (2aii) and (2bii) are analogous.
Rotate $\ga$ so that the tangent line to $\ga(w_*)$ is the $x$--axis and $\Tau(w_*) = 1 \in S^1 \subset \C$.
Angenent \cite{angenent1991parabolic} tells us that $\ga(w_*)$ must be an isolated intersection point of the $x$--axis and $\ga$, and, in particular, that $\ga(u)$ must either lie strictly above the $x$--axis for all $u < w_*$ sufficiently close, or lie strictly below the $x$--axis for all $u < w_*$ sufficiently close. 
We will show that it lies below.

Suppose for a contradiction that $\ga(u)$ lies strictly above the $x$--axis for all $u < w_*$ sufficiently close. 
Fix any $w_1 < w_*$ with $\ga(w_1)$ lying strictly above the $x$--axis, and let $L$ be the radial line through $\ga(w_1)$. 
Invoking Angenent \cite{angenent1991parabolic} again, we have that $\ga $ must intersect $L$ over $[w_1, w_*]$ discretely, and so, by increasing $w_1<w_*$ if necessary, we may assume that $\ga$ does not intersect $L$ over $(w_1,w_*)$. 
Hence adjoining the arc of $\ga$ over $[w_1, w_*]$ to the radial lines from the origin $0$ to $\ga(w_1)$ (which lies above the $x$--axis) and from $\ga(w_*)$ to $0$ produces an embedded closed curve oriented anti-clockwise.  
Thus the negation of the area enclosed by the curve, which by Remark \ref{ca_for_closed_curve} coincides with the swept area $\ca(w_1, w_*)$, is strictly \emph{negative}, contradicting (2i) from Claim \ref{cl:sweptarea}.
See Figure \ref{fig:wthdoicallthis?} for a schematic.

\begin{figure}
\centering
\begin{tikzpicture}[scale=1.3]
\draw[fill, blue!8] (-2, 0) to[out=180,in=-60] (-4, 2) to (0,0) to (-2,0);
\draw[very thick, <-] (-2, 0) to[out=180,in=-60] (-4, 2) to[out=180-60, in=0] (-5, 3) to[out=180, in=60] (-6,2) to[out=180+60, in=0] (-7,1) to[out=180, in=-70] (-8, 2.5);

\draw[dashed] (-8,0) node[left]{$x$--axis} to (-2,0) ; 
\draw[dashed] (-4, 2) to (-8, 4) node[left]{$L$} ;
\draw[dashed, very thick] (-4, 2) to (0,0) node[below]{$0$} to (-2,0);
\draw[dashed] (1, -.5) to (0,0) to (1, 0) ;

\node at (0,0) [circle,fill,inner sep=1.1pt]{} ;
\node at (-2,0) [circle,fill,inner sep=1.1pt]{} ;
\node at (-4,2) [circle,fill,inner sep=1.1pt]{} ;
\draw (-2, 0) node[below]{$\ga(w_*)$} ;
\draw (-4,2) node[above right]{$\ga(w_1)$} ;

\draw[very thick, ->] (-2,1) to (-2.2, 1.1) ;
\draw[very thick, ->] (-1,0) to (-0.8, 0) ;
\end{tikzpicture}
\caption{Schematic to show that if the tangent line to $\ga(w_*)$ is the $x$--axis and $\Tau(w_1)$ points inward to the origin for $w_1 < w_*$ sufficiently close, but $\ga(w_1)$ lies above the $x$--axis for all $w_1 < w_*$ sufficiently close, then necessarily $\ca(w_1, w_*) < 0$ for some $w_1 < w_*$, and therefore 
$(v_*, w_*)$ cannot maximise $\ca$ 
by Claim \ref{cl:sweptarea}.}
\label{fig:wthdoicallthis?}
\end{figure}
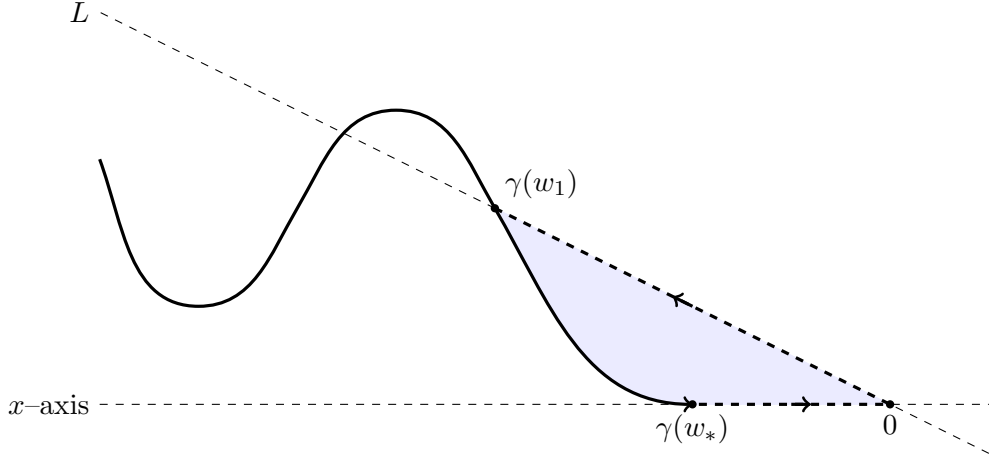

Hence $\ga(u)$ lies below the $x$--axis for all $u < w_*$ sufficiently close. 
Write $\ga = x + iy$ in Cartesian coordinates, so that $y(w_1) < y(w_*) = 0$ for all $w_1 < w_*$ sufficiently close. 
As $\Tau(w_*) = 1$, meaning here that $x_u(w_*) > 0$, we must have that $x(w_1) < x(w_*)$ for all $w_1 < w_*$ sufficiently close. 
Thus $\ga(w_*)$ is upward and rightward from $\ga(w_1)$ for all $w_1 < w_*$ sufficiently close. 
This enough to conclude \eqref{winback}. 

Finally, similar to in Claim \ref{cl:sweptarea}, we show that (2) implies (1) using the symmetries from Remark \ref{symmetry_rmk}. 
(The argument below is reversible, so they are in fact equivalent.)
By rotating, we may assume without loss of generality that $\ga(v_*)$ lies on the $x$--axis.
Let $\bar\ga$ be the reflection of $\ga$ in the $x$--axis and $\widetilde\ga(u) = \hat{\bar\ga}(u)  = \bar{\ga}(1 - u) $ be the reflected and reversed curve, and let $\widetilde v_* = 1 - w_*$ and $\widetilde w_* = 1- v_*$.
Let 
 $\widetilde \ca$ be the swept area of $\widetilde \ga$,
so that  $(\widetilde v_*, \widetilde w_*) $ maximises $\widetilde \ca$. 

By the first assertion of (2), we have that the tangent line to $\widetilde\ga$ at $\widetilde\ga(\widetilde w_*) = \ga(v_*)$ is the $x$--axis, which is unaffected by reversing and reflecting $\widetilde\ga$ back to $\ga$. 
Let $\bar\Tau$ and $\bar\th_{v_*}$ be the tangent vector and winding function of $\bar\ga$ respectively, and let $\widetilde\Tau$ and $\widetilde\theta_{\widetilde w_*}$ be the tangent vector and winding function of $\widetilde\ga$ respectively.
If $\Tau(u)$ points outward from the origin, then so will $\bar\Tau(u)$, and therefore $\widetilde \Tau(1-u)$ points \emph{inward} to the origin. 
Similarly, if $\Tau(u)$ points inward to the origin, then $\widetilde\Tau(1-u)$ points outward from the origin.
Since the inequalities $v_1 < v_* < v_2$ transform to $\widetilde w_1 < \widetilde w_* < \widetilde w_2$, where $\widetilde w_1 = 1 - v_2$ and $\widetilde w_2 = 1- v_1$, we thus have that the setting of (1ai) transforms to the setting of (2bii) (with $\widetilde \Tau(\widetilde w_2)$ for $\widetilde w_2 > \widetilde w_*$), of (1aii) to (2bi), of (1bi) to (2aii), and of (1bii) to (2ai).
As $$\widetilde\th_{\widetilde w_*}(u) = \bar\th_{v_*}(1 - u) = -\th_{v_*}(1-u)$$  for all $u \in (0,1)$,
we deduce that $\eqref{winfor}$ (with $\widetilde \th_{\widetilde w_*}(\widetilde w_2)$ for $\widetilde w_2 > \widetilde w_*$) implies \eqref{voutback}, \eqref{winback} implies \eqref{voutfor}, \eqref{woutfor} implies \eqref{vinback}, and \eqref{woutback} implies \eqref{vinfor}.
Hence (2) implies (1), as asserted.
\end{proof}

We now combine the local control on the winding functions from Claim \ref{cl:windingfuncsversion2} with the global control on swept areas from Claim \ref{cl:sweptarea} to deduce \emph{global} control on the winding functions.
For consistency, we keep the same numbering in the global control Lemma \ref{lem:windingconstraints} below as in the local control Claim \ref{cl:windingfuncsversion2} above: entering/exiting [(i)/(ii)] $\ga(v_*)$/$\ga(w_*)$ [(1)/(2)] in an outward/inward direction [(a)/(b)].
And again, while we do not require either (1ai) or (2bii) in the proof of Lemma \ref{lem:turnangbddred}, we include them here for completeness.

\begin{lem}\label{lem:windingconstraints}
In the setting of Lemma \ref{lem:turnangbddred}, suppressing the dependence on $t$ as above,
    \begin{compactenum}[(1)]
        \item suppose $v_* > 0$, and let $\th_{v_*}$ be the winding function of $\ga$. 
        \begin{compactenum}[(a)] 
            \item \begin{compactenum}[(i)]
                \item If $\Tau(v_1)$ points outward (from the origin) for all $v_1 < v_*$ sufficiently close, then
                \beq\label{thvoutback} \th_{v_*}(v_1) < \pi \quad \text{for all }\, 0 < v_1 < v_*. \eeq
                \item If $\Tau(v_2)$ points outward for all $v_2 > v_*$ sufficiently close, then
                 \beq\label{thvoutfor} \th_{v_*}(v_2) < 0 \quad \text{for all }\, v_* < v_2 \leq w_*. \eeq
            \end{compactenum}
             \item \begin{compactenum}[(i)]
                \item If $\Tau(v_1)$ points inward (to the origin) for all $v_1 < v_*$ sufficiently close, then
                \beq\label{thvinback} \th_{v_*}(v_1) < 0 \quad \text{for all }\, 0 < v_1 < v_*. \eeq
                \item If $\Tau(v_2)$ points inward for all $v_2 > v_*$ sufficiently close, then 
                 \beq\label{thvinfor} \th_{v_*}(v_2) < \pi \quad \text{for all }\, v_* < v_2 \leq w_*. \eeq
            \end{compactenum}
        \end{compactenum}
        \item suppose $w_* < 1$, and let $\th_{w_*}$ be the winding function of $\ga$.  
        \begin{compactenum}[(a)]
            \item \begin{compactenum}[(i)]
                \item If $\Tau(w_1)$ points outward   for all $w_1 < w_*$ sufficiently close, then
                \beq\label{thwoutback} \th_{w_*}(w_1) > -\pi \quad \text{for all }\, v_* \leq  w_1 < w_*. \eeq
                \item If $\Tau(w_2)$ points outward   for all $w_2 > w_*$ sufficiently close, then 
                 \beq\label{thwoutfor} \th_{w_*}(w_2) > 0 \quad \text{for all }\, w_* < w_2 < 1. \eeq
            \end{compactenum}
             \item \begin{compactenum}[(i)]
                \item If $\Tau(w_1)$ points inward   for all $w_1 < w_*$ sufficiently close, then
                \beq\label{thwinback} \th_{w_*}(w_1) > 0 \quad \text{for all }\, v_* \leq w_1 < w_*. \eeq
                \item If $\Tau(w_2)$ points inward for all $w_2 > w_*$ sufficiently close, then  
                 \beq\label{thwinfor} \th_{w_*}(w_2) > -\pi \quad \text{for all }\, w_* < w_2 < 1. \eeq
            \end{compactenum}
        \end{compactenum}
    \end{compactenum}
\end{lem}

 When $v_* > 0$ or $w_* < 1$, Claim \ref{cl:windingfuncsversion2} tells us that the tangent line at $\ga(v_*)$ or at $\ga(w_*)$ respectively passes through the origin. 
 Therefore Lemma \ref{lem:windingconstraints} exhausts all possible configurations of $\ga$.

Before giving the proof of Lemma \ref{lem:windingconstraints}, we illustrate the core idea of the argument by means of a simple example.

For $v, w \in \R$ with $v < w$, let the smooth arc $\ga : [v,w] \to \C \simeq \R^2$ be formed by  joining $i =: \ga(v)$ to $2i$ and joining $2i$ to $-1+2i =: \ga(w)$ as in Figure \ref{fig:badconfigeg}.
Let $\Tau$ be the tangent vector, $\ca$ be the swept area and $\th_{v}$ be the winding function of $\ga$ respectively.

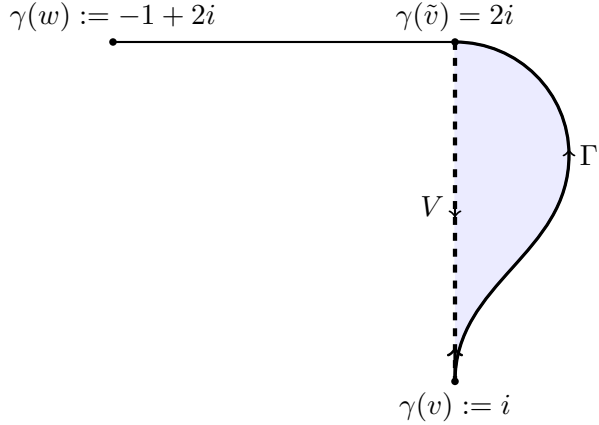
\begin{figure}
\centering
\begin{tikzpicture}[scale=1.5]
\draw[fill, blue!8] (0,1)  to[out=90,in=-90] (1,3) to[out=90, in=0] (0, 4) to[out=-90,in=90] (0, 2) ;
\draw[thick] (0,1) node[below]{$\ga(v) := i$} to[out=90,in=-90] (1,3) node[right]{$\Ga$} to[out=90, in=0] (0, 4) node[above]{$\ga(\tilde v) = 2i$} to[out=180,in=0] (-3, 4) node[above]{$\ga(w) := -1 + 2i$};
\draw[very thick] (0,1) to[out=90,in=-90] (1,3) to[out=90, in=0]  (0, 4) ;
\draw[ultra thick, dashed] (0,4) to (0,1); 
\draw[->, very thick] (0,1) to (0, 1.3) ;
\draw[->, thick] (1,2.95) to (1, 3.05) ;
\draw[->, thick] (0,2.55) node[left]{$V$} to (0, 2.45) ;
\node at (0,1) [circle,fill,inner sep=1pt]{};
\node at (0,4) [circle,fill,inner sep=1pt]{};
\node at (-3,4) [circle,fill,inner sep=1pt]{};
\end{tikzpicture}
\caption{Example of a configuration of a curve $\ga : [v, w] \to \C $ for which \eqref{voutfor} with $v_* = v$ holds but  \eqref{thvoutfor} with $v_* = v$ and $v_2 = \tilde v$ does \emph{not} hold at some $v < \tilde v < w$.
A simple area argument shows that $\ca(v, \tilde v) < 0$.}
\label{fig:badconfigeg}
\end{figure}

From Figure \ref{fig:badconfigeg} we can see that $\Tau(v) = i \in S^1 \subset \C$ is outward-pointing.
By continuity, we have that $\Tau(u)$ is outward-pointing for all $u > v$ sufficiently close, i.e.~for $v$ in place of $v_*$, we are in the common setting of (1aii) of Claim \ref{cl:windingfuncsversion2} (which tells us that $\th_{v}(u) < 0$ for $u > v$ sufficiently close) and Lemma \ref{lem:windingconstraints}.
We also have that $\th_{v}(w) = \frac\pi4 \geq 0$, i.e.~the conclusion \eqref{thvoutfor} of (1aii) of Lemma \ref{lem:windingconstraints} with $v_* = v$ does \emph{not} hold, which can be seen by noting that 
$$\ga(w) = -1 + 2i = i + (1+i)i = \ga(v) + \sqrt{2} \cdot e^{\frac{i\pi}4} \cdot \Tau(v). $$ 
Thus by continuity there is a point $v < \tilde v \leq w$ for which $\th_{v}(\tilde v) = 0$,
which is forcing $\ga(\tilde v)$ to lie on the same line radial line as $\ga(v) = i$, i.e.~on the imaginary axis.
By inspecting Figure \ref{fig:badconfigeg}, we see that $\ga(\tilde v) = 2i$. 

Let $V$ be the vertical line segment from $2i = \ga(\tilde v)$ to $i = \ga(v)$, illustrated as the dashed line in Figure \ref{fig:badconfigeg}, and let $\Ga$ be closed curve formed by adjoining $V$ to the arc of $\ga$ from $i$ to $2i$. 
By additivity, the swept area of $\Ga$ is the sum of $\ca(v, \tilde v)$ and the swept area of $V$.
It is easy to see\footnote{for example, by computing that $V_u \wedge V = 0$ for any regular parametrisation $u \mapsto V(u)$} that the swept area of $V$ vanishes, and therefore $\ca(v, \tilde v)$ coincides with the swept area of $\Ga$.
As $\Ga$ is a piecewise smooth closed curve oriented anti-clockwise, Remark \ref{ca_for_closed_curve} tells us that the swept area of $\Ga$ is the \emph{negation} of the area of the region enclosed  by $\Ga$, pictured as the shaded blue region. 
Hence we have that
$$\begin{aligned}
\ca(v, \tilde v) &= \text{swept area of $\Ga$} \\
&= - [ \text{area of shaded blue region} ] \\
&< 0.    
\end{aligned}$$
In the proof of Lemma \ref{lem:windingconstraints}, the argument above will form part of a contradiction argument. 
For example, if $(v_*, w_*)$ maximises $\ca$, then we know by (1aii) of Claim \ref{cl:sweptarea} that $\ca(v_*, v_2) \geq 0$.
If the arc of $\ga$ between $\ga(v_*)$ and $\ga(w_*)$ looked as in Figure \ref{fig:badconfigeg}, then by the argument above we would conclude that $\ca(v_*, v_2) < 0$ for some $v_* < v_2 \leq w_*$, which is contradictory.
So the argument above \emph{rules out} configurations such as that in Figure \ref{fig:badconfigeg}.
In practice we need to rule out much wilder arrangements than above, and so we will need a more general argument, but the core idea remains the same.

\begin{proof}[Proof of Lemma \ref{lem:windingconstraints}] 
As demonstrated in Claims \ref{cl:sweptarea} and \ref{cl:windingfuncsversion2}, by symmetry it suffices to prove\footnote{Specifically, the following pairs of parts are equivalent: (1ai) and (2bii); (1aii) and (2bi); (1bi) and (2aii); (1bii) and (2ai).} only one of (1) or (2). 
We give the proof of (1bii); the arguments for (1ai), (1aii) and (1bi) are analogous.

By rotating $\ga$, we may assume without loss of generality that $p_* := \ga(v_*)$ lies on the non-positive $x$--axis and $\Tau(v_*) = 1 \in S^1 \subset \C$.  
Then automatically $\ga(u)$ lies above the $x$--axis for all $u > v_*$ sufficiently close.
Suppose for a contradiction that \eqref{thvinfor} does not hold, i.e.~that $\th_{v_*}(v_2) = \pi$ for some $v_* < v_2 \leq w_*$. 
By reducing $v_2$ if necessary, we may assume that $\th_{v_*}(u) < \pi$ for all $v_* < u < v_2$.
Since $\Tau(v_*) = 1$, it follows that $p_2 := \ga(v_2)$ lies to the left of $p_*$ on the negative $x$-axis. 
Adjoin the (radial) line segment from $p_2 = \ga(v_2)$ to $p_* = \ga(v_*)$  to the arc of $\ga$ from $p_*$ to $p_2$. 
If the resulting closed curve $\Ga$ is embedded, then we may use Remark \ref{ca_for_closed_curve} directly and deduce that $\ca(v_*, v_2) < 0$, which by (1bii) of Claim \ref{cl:sweptarea} gives a contradiction. 
However, in general the curve may intersect itself along the $x$--axis, as in Figure \ref{badcon:fig2}.

\begin{figure}
\centering
\begin{tikzpicture}[scale=1.6, rotate=180]
\draw[fill, blue!8] (3,0) to[out=180, in=90] (2.75, -.43) to[out=-90,in=0] (2.5,-1) to[out=180,in=-90] (2,0) to[out=90, in=180] (3,1) to[out=0,in=90] (4,0) to[out=-90, in=180] (4.5,-1) to[out=0,in=-90] (5, 0) to[out=90,in=0] (3,2) to[out=180,in=90] (1, 0)  to[out=-90,in=180] (3.5,-2) to[out=0,in=-90] (6, 0) ;
\fill[pattern = crosshatch, pattern color=blue!24] (2.75, -.43) to[out=-30, in=210]  (3.25, -.43) to[out=30, in=-90] (3.5,0) to (3,0) to[out=180, in=90] (2.75, -.43);
\draw[fill, blue!16]  (4,0) to[out=-90, in=180] (4.5, -1) to[out=0,in=-90] (5, 0) to (4,0) ;
\draw[very thick] (3,0) to[out=180, in=90] (2.75, -.43) to[out=-90,in=0] (2.5,-1) to[out=180,in=-90] (2,0) to[out=90, in=180] (3,1) to[out=0,in=90] (4,0) to[out=-90, in=180] (4.5, -1) to[out=0,in=-90] (5, 0) to[out=90,in=0] (3,2) to[out=180,in=90] (1, 0)  to[out=-90,in=180] (3.5,-2) to[out=0,in=-90] (6, 0) ;
\draw[dashed, very thick] (3,0) to (6,0) ;
\draw[dashed] (0.5,0) node[right]{$0$} to (3,0);
\node at (.5,0) [circle,fill,inner sep=1pt]{} ;
\draw[dashed] (6, 0) to (6.75,0) ;
\draw[->, thick] (3,0) to (2.9,0);
\draw[->, thick] (3,1) to (3.1,1);
\draw[->, thick] (3,2)  to (2.9,2);
\draw[->, thick] (4.5,-1) to (4.6,-1);
\draw[->, thick] (3.5,-2) node[above]{$\Ga$} to (3.6,-2);
\draw[->, thick] (2.5,-1) to (2.4, -1);
\draw[->, thick] (4.5, 0) to (4.4,0) ;
\draw (3,0) node[below]{$p_*$} ;
\draw (6,0) node[below]{$p_2$} ;
\draw (3.65,0) node[below]{$p_1$} ;
\draw (2.75,-.5) node[right]{$p_0$} ;
\node at (6,0) [circle,fill,inner sep=1.1pt]{} ;
\node at (3,0) [circle,fill,inner sep=1.1pt]{} ;
\draw[dotted, samples=60,domain=0:240] plot ({3+cos(\x)/2}, {sin(\x)/2});
\node at (2.75,-.43) [circle,fill,inner sep=1.1pt]{} ;
\node at (3.5,0) [circle,fill,inner sep=1.1pt]{} ;
\draw[ultra thick, dotted, samples=30,domain=240:360] plot ({3+cos(\x)/2}, {sin(\x)/2});
\end{tikzpicture}
\caption{Illustration of a less transparently bad configuration of $\Ga$, the closed curve formed by adjoining the radial line segment from $p_2$ to $p_*$ to the arc of $\ga$ between $p_*$ and $p_2$. 
We may split $\Ga$ into the portion inside the dotted ball around $p_*$ and the portion outside the ball, and close each portion with the appropriately-orientated circular arc along the boundary of the dotted circle (between $p_0$ and $p_1$) to get two closed anti-clockwise-oriented curves $\Ga_1$ and $\Ga_2$ respectively.}
\label{badcon:fig2}
\end{figure}
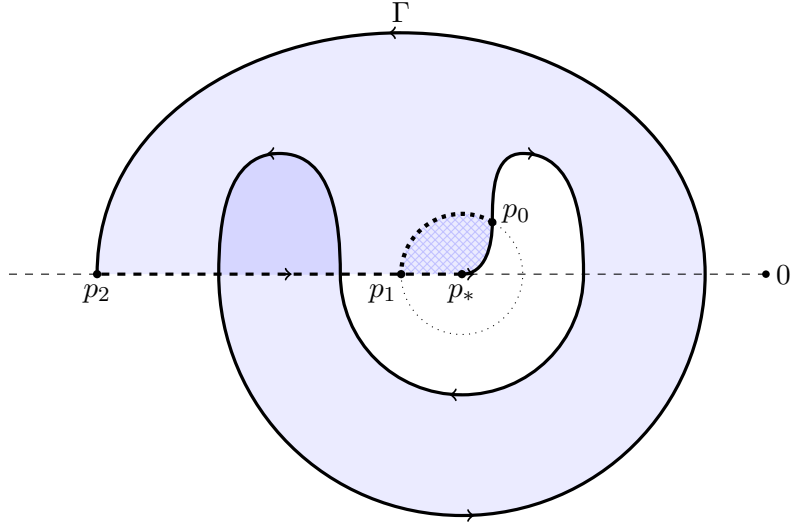

We argue that $\ca(v_*, v_2) < 0$ still holds in this case. 
Choose $\de > 0$ so small that the portion of $\Ga$ inside the ball $B_\de(p_*)$ of radius $\de$ around $p_*$ is topologically  embedded (it will in fact be piecewise smooth), and so small that $\Ga$ intersects the boundary of $B_\de(p_*)$ exactly twice, once on the arc of $\ga$, at the point $p_0$, and once on the line segment from $p_*$ to $p_2$, at the point $p_1$.
In particular, the portion of $\Ga$ inside the ball is a connected arc, which may be closed by traversing $\partial B_\de(p_*)$ anti-clockwise from $p_0$ to $p_1$.
The resulting closed curve $\Ga_1$ is piecewise smooth, embedded and oriented anti-clockwise, and so 
by Remark \ref{ca_for_closed_curve} the swept area of $\Ga_1$ is strictly negative.
Let $\Ga_2$ be the closed curve formed by adjoining the remaining portion of $\Ga$ to the clockwise circular arc along  $\partial B_\de(p_*)$ from $p_1$ to $p_0$. 
Since the swept area of $\Ga$, which coincides with $\ca(v_*, v_2)$, is the sum of the swept areas of $\Ga_1$ and $\Ga_2$, it suffices to show that swept area of $\Ga_2$ is negative.

Lift $\Ga_2$ to a curve $\ti\Ga_2$ in the universal cover of $\C\setminus \{p_*\}$, which has a global complex coordinate
$w\in\C$ satisfying $z-p_*=e^w$.
The lift $\ti\Ga_2$ is now an \textit{embedded} closed curve in the $w$-plane, so is the boundary of a domain $\Om$ in the $w$-plane.
By Remark \ref{ca_for_closed_curve} 
we can control the swept area of $\Ga_2$ by
\beqa
-\frac{1}{2i}\int_{\Ga_2} \zb\, dz &= -\frac{1}{2i}\int_{\ti\Ga_2} \overline{(p_*+e^w)}e^w \, dw
= -\frac{1}{2i}\int_{\ti\Ga_2} e^{\overline{w}+w} \, dw\\
&= -\frac{1}{2i}\int_{\Om} e^{\overline{w}+w} \, d\bar{w}\wedge dw
=-\int_{\Om} e^{2\Re(w)} \, dx\wedge dy\\
&<0,
\eeqa
where we have written $w=x+iy$, so that $d \bar w \wedge dw = 2i \, dx \wedge dy$, and used Cauchy's theorem (to get rid of the $p_*$) and Stokes' theorem.

Hence the swept area of $\Ga$, and thus $\ca(v_*, v_2)$, is strictly negative;  
but this contradicts (1ii) of Claim \ref{cl:sweptarea}. 
 Therefore we must have had that $\th_{v_*}(v_2) <\pi$ for all $v_* < v_2 \leq w_*$, i.e.~must have had \eqref{thvinfor}.
\end{proof}

Using Lemma \ref{lem:windingconstraints}, we can now complete the proof of Lemma \ref{lem:turnangbddred}.

\begin{proof}[Proof of Lemma \ref{lem:turnangbddred}]
We start by proving a few subcases.

\begin{claim}\label{cl:voutfor}
    \begin{enumerate}[(1)]
        \item  If $v_*>0$ and $\Tau(v_2)$ points outward for all $v_2 > v_*$ sufficiently close, i.e.~if we are in the setting of (1aii) of Lemma \ref{lem:windingconstraints}, then $\Psi(v_*, w_*) \leq 0$.
        \item If $w_*<1$ and $\Tau(w_1)$ points inward for all $w_1 < w_*$ sufficiently close, i.e.~if we are in the setting of (2bi), then $\Psi(v_*, w_*) \leq 0$.
    \end{enumerate}
\end{claim}

\begin{proof}
    By symmetry, it suffices to prove only (2).
    From (2bi) of Lemma \ref{lem:windingconstraints}, we have that \eqref{thwinback} holds, i.e.~that
\beql{gurriers}
\th_{w_*}(w_1) > 0\text{ for all } v_* \leq w_1 < w_*, \, \text{ in particular, } \, \th_{w_*}(v_*) > 0.
\eeq
    In the case $v_* = 0$, by (i) of Corollary \ref{cor:dC} we deduce that
    $$ \Psi(v_*, w_*) = \Psi(0, w_*) = -\lim_{w_1 \downto 0} {\th_{w_*}(w_1)} \leq 0,$$
    as required.
    
    In the case $v_* > 0$, by Claim \ref{rmk:windingfunc} we know that either the setting of (1aii) or the setting of (1bii) of Lemma \ref{lem:windingconstraints} holds.
    If we are in the setting of (1aii), then \eqref{thvoutfor} holds. 
    In particular, $\th_{v_*}(w_*) < 0$, and from Lemma \ref{lem:dC} and \eqref{gurriers} we conclude that
    $$ \Psi(v_*, w_*) = \th_{v_*}(w_*) - \th_{w_*}(v_*) < 0.$$
    If instead (1aii) does not hold, then we are in the setting of (1bii), so \eqref{thvinfor} holds, so in particular 
    \beql{guess1th} \th_{v_*}(w_*) <  \pi \eeq
    and, by Lemma \ref{lem:dC} and \eqref{gurriers}, we have only
    \beql{coarsePsi} \Psi(v_*, w_*) = \th_{v_*}(w_*) - \th_{w_*}(v_*) < \pi. 
    \eeq
    Suppose that \beql{badPsi} 0 < \Psi(v_*, w_*) < \pi .\eeq
    By Claim  \ref{rmk:windingfunc}, as $v_* > 0$ and the setting of (1aii) does \emph{not} hold, 
    we have $\ga(v_*) \neq 0$ and $\Tau(v_*)$ points inward to the origin.
    Therefore we can rotate $\ga$ so that $\ga(v_*)$ lies on the negative $x$--axis and $\Tau(v_*) = 1 \in S^1 \subset \C$. 
    Since $0 < \Psi(v_*, w_*) < \pi$, we must have that $\Tau(w_*)$ points into the upper half plane.
        This implies that $\ga(w_*)$ cannot lie in the upper half plane, because if it did, then as $w_* < 1$ and we are in the setting of (2bi) with $\ga(w_*) \neq 0$, by Claim \ref{cl:windingfuncsversion2} the tangent  $\Tau(w_*)$ would point inward to the origin, i.e~would point into the \emph{lower} half plane.

    By Definition \ref{def:incfunc}, we have 
    $$ e^{i \th_{v_*}(w_*)} = e^{i\th_{v_*}(w_*)} \Tau(v_*) = \frac{ \ga(w_*) - \ga(v_*) }{ | \ga(w_*) - \ga(v_*) | } .  $$ 
    Taking imaginary parts gives
    $$ \sin \th_{v_*}(w_*) \leq 0,$$
    because $\ga(w_*)$ has nonpositive imaginary part and $\ga(v_*)$ is real.
    Combined with \eqref{guess1th}, we get the refined estimate
    \beql{gess2th} \th_{v_*}(w_*) \leq 0. \eeq
    Using \eqref{gess2th} instead of \eqref{guess1th} in Lemma \ref{lem:dC}, together with $\th_{w_*}(v_*) > 0$  from \eqref{gurriers}, gives
    $$ \Psi(v_*, w_*) = \th_{v_*}(w_*) - \th_{w_*}(v_*) < 0.$$
    But this contradicts \eqref{badPsi}, strengthening \eqref{coarsePsi} to 
    \[ \Psi(v_*, w_*) \leq 0, \]
    as required.
\end{proof}

\begin{claim}\label{cl:vinback}
Suppose $w_*<1$ and  $\Tau(w_2)$ points outward for all $w_2 > w_*$ sufficiently close, i.e.~we are 
in the setting of (2aii) of Lemma \ref{lem:windingconstraints}.
If either $v_* = 0$, or $v_* > 0$ and $\Tau(v_1)$ points inward for all $v_1 < v_*$ sufficiently close, i.e.~the setting of (1bi) holds, then $\Psi(v_*, w_*) \leq \alpha.$
\end{claim}

\begin{proof}
    From (2aii) of Lemma \ref{lem:windingconstraints}, we have that \eqref{thwoutfor} holds, i.e.~$\th_{w_*}(w_2) > 0$ for all $w_* < w_2 < 1$. 
    So by (ii) of Corollary \ref{cor:dC},
    \beql{w_star}
    \Psi(w_*, 1) = \lim_{w_2 \upto 1} {\th_{w_*}(w_2)} \geq 0.
    \eeq
    By symmetry, we immediately have that if $v_*>0$ and $\Tau(v_1)$ points inward for all $v_1 < v_*$ sufficiently close, then
    \beql{v_star}
    \Psi( 0, v_*) \geq 0.
    \eeq
    If $v_* = 0,$ then by the additivity of the swept area we deduce by \eqref{w_star} that
    $$\Psi(v_*, w_*) = \Psi(0, w_*) = \Psi(0,1) - \Psi(w_*, 1) \leq \alpha.$$
    Otherwise, $v_* > 0$ and $\Tau(v_1)$ points inward for all $v_1 < v_*$ sufficiently close, so we have both \eqref{w_star} and \eqref{v_star}, and again by additivity we deduce that
    \[ \Psi(v_*, w_*) = \Psi(0, 1) - \Psi(0, v_*) - \Psi(w_*, 1) \leq \alpha . \qedhere \]
\end{proof}

To conclude the proof of Lemma \ref{lem:turnangbddred}, we explain how Claims \ref{rmk:windingfunc}, \ref{cl:voutfor} and \ref{cl:vinback} can be combined to show \eqref{turnangbddred} for all possible configurations of $\ga$.

If $v_* = 0$ and $w_* = 1$, then \eqref{turnangbddred} follows by Lemma \ref{boundary_lem}.
Otherwise, $v_* > 0$ or $w_* < 1$ (or both). 
By reflecting and reversing $\ga$, cf.~Remark \ref{symmetry_rmk}, we may assume without loss of generality that $w_* < 1$.
If the setting of (2bi) of Lemma \ref{lem:windingconstraints} holds, or if $v_* > 0$ and the setting of (1aii) holds, then \eqref{turnangbddred} follows by Claim \ref{cl:voutfor}. 
Otherwise, we are neither in setting (2bi), nor in setting (1aii) with $v_* > 0$. 
As $w_* < 1$, by Claim \ref{rmk:windingfunc} we know that the setting of (2aii) holds.
Similarly using Claim \ref{rmk:windingfunc}, either $v_* = 0$, or $v_* > 0$ and the setting of (1bi) holds.
These two remaining cases are resolved by Claim \ref{cl:vinback}. 
\emph{The proof of Lemma \ref{lem:turnangbddred} is now complete}.
\end{proof}

\section{Relating the Harnack inequalities}
\label{Harnack_relating_sect}

The purpose of this section is to relate Hamilton's Harnack inequality to the alternative Harnack inequality of Theorem \ref{main_new_harnack_thm}. We do that by considering the pointwise curvature estimate for curve shortening flow in terms of the so-called support function that each implies and finding that they agree in the case of convex flows.

We will consider the alternative Harnack inequality only in the case of polar graphical flows but we will work directly from Theorem \ref{main_new_harnack_thm} rather than from Theorem \ref{wedge_harnack_thm1} because 
it is more convenient to work with arc-length parametrisation rather than parametrisation by the polar angle $\vph$.
In this case we can take the inequality 
\beql{could_be_equality1}
\ca(v,w,t)-t\Psi(v,w,t)=\h(v,w,t)\geq 0,
\eeq
fix $v$ and take a one-sided derivative with respect to $w$ at $v$, where we have equality, to give 
$$\pl{}{w}\bigg|_{w=v}\big[\ca(v,w,t)-t\Psi(v,w,t)\big]\geq 0,$$
that is, 
\beql{could_be_equality2}
-\half \langle \ga,\nu \rangle - t\ka \geq 0 
\eeq
at each point along the curve,
where we have used \eqref{simpler_A_exp} and the fact that curvature is the rate of change of tangent angle with respect to arc-length. 
The \emph{support function} 
\beql{D_def}
\cd(u):=\langle \ga,-\nu \rangle
\eeq
can be interpreted geometrically as the shortest distance from the origin to the tangent line at $\ga(u)$, as illustrated in Figure \ref{D_def_fig}, where we use $\vph$ as the parameter $u$. 
By similarly considering the case of the $\be$-wedge solution for any $\be \in (0, \pi)$, when have equality throughout in \eqref{could_be_equality1}, 
and hence also in \eqref{could_be_equality2}, we obtain the following simple curvature estimate. 

\begin{figure}
\centering
\begin{tikzpicture}[scale=1]

\draw [->] (0,0) -- (-2.8*1,2.8*1.732);
\draw [->] (0,0) -- (-2.8*2,0);

\draw[very thick, samples=40,domain=1:58, ->] 
plot ({-cos(\x)*(\x/100+1.1+2.5/(1+0.01*(\x-30)^2))/(sin(3*\x))^(0.5)}, {sin(\x)*(\x/100+1.1+2.5/(1+0.01*(\x-30)^2)))/(sin(3*\x))^(0.5)} ) node[right]{$\ga(\cdot,t)$};


\draw[dashed, thick] (-3.6,-1) -- (-0.1,3);
\draw[dashed, thick] (-0.85,-1) -- (-0.85+3.5*0.8,-1+4*0.8);
\draw[<->, thick] (-0.45,2.6) -- node[above right, pos=0.5]{$\cd(\vph,t)$} (-0.45+1.55,2.6-0.875*1.55);

\filldraw (-2.33,0.45) circle (1.5pt) node[above left]{$\ga(\vph,t)$};

\draw[-, dotted] (-2.33,0.45) -- (0,0);

\draw[<->, thick] (-1.3,0) arc (180:162:0.8);
\draw[-] (-1.3,0.15)  node[left]{$\vph$};

\end{tikzpicture}
\caption{Illustration of $\cd(\vph,t)$.}
\label{D_def_fig}
\end{figure}
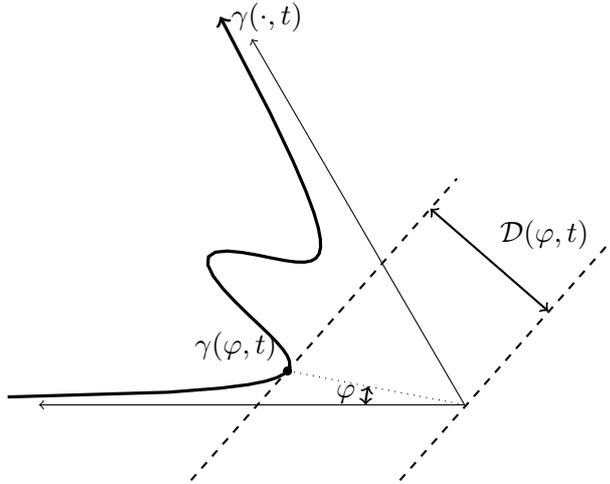

\begin{cor} 
\label{curv_cor}
Any polar graphical solution to curve shortening flow with initially radial ends 
as in Theorem \ref{wedge_harnack_thm1} must satisfy
\beql{curv_cor_est}
\ka(\vph,t)\leq \frac{1}{2t} \cd(\vph,t)
\eeq
for $t>0$. 
For any $\be \in (0, \pi)$, the  $\be$-wedge solution achieves  equality in \eqref{curv_cor_est}.
\end{cor}

This estimate is known already even in higher dimensions; see the work of Chodosh, Choi, Mantoulidis and Schulze \cite{CCMS}.
The relevance here is that it is exactly the curvature estimate independent of $\be$ that one obtains from Hamilton's Harnack estimate in the convex case.
To establish that, we return to the claim in the introduction that Hamilton's Harnack inequality \eqref{ham_harn_1d} is equivalent to the inequality 
$(\ka t^\half)_t\geq 0$, where the time derivative is with $\psi$ fixed. Although this is true for general convex uniformly proper solutions over general time intervals (that is, for convex uniformly proper solutions defined for times $t \in [0, T)$, where $T \in (0 , \infty]$, whose ends are not necessarily initially radial)
with a minor modification of the proof below, and even generalises to higher dimensions, we give a precise statement and proof in the special case that we require.

\begin{thm}[Reformulation of Hamilton's Harnack inequality]
\label{Harnack_psi}
Any convex uniformly proper  solution $\ga:\ci\times [0,\infty)\to\R^2$ to curve shortening flow with initially radial ends, 
other than the static flow of a straight line, can be viewed for $t > 0$ (after applying an ambient isometry and reparametrising at each time) as a map $\ga:(0,\pi-\be)\times (0,\infty)\to\R^2$, 
for some $\be\in (0,\pi)$,
parametrised as $\ga(\psi,t)$ where $\psi$ is the tangent angle, having geodesic curvature $\ka(\psi,t)>0$ satisfying
\beql{Harnack_ineq_psi}
(\ka t^\half)_t\geq 0,
\eeq
where the time derivative is with $\psi$ fixed.
\end{thm}

\begin{proof}
The flow $\ga$ has bounded curvature over compact time intervals as a consequence of its initial radial ends, so since $\ga(\cdot, t)$ is not a straight line, the initial convexity $\ka(\cdot, 0) \geq 0$ together with the strong maximum principle applied to
 the standard evolution equation 
\[ \ka_t=\ka_{uu}+\ka^3 \]
for the curvature (from e.g.~\cite{gage_hamilton}), where $u$ is arc-length parametrisation at the given time, ensure that $\ka(\cdot, t) > 0$ whenever $t > 0$.
Thus Hamilton's Harnack inequality \eqref{ham_harn_1d} applies to $\ga$, and can be written as
$$\ka_{uu}+\ka^3 +\frac{\ka}{2t}\geq \frac{\ka_u^2}{\ka}.$$

As $\ga(\cdot , 0)$ is convex but is not a straight line, by applying an ambient isometry we may assume that the initially radial ends are at
angles $\pi$ and $\pi-\be$ for some $\be\in (0,\pi)$.
By the smoothness and the strict convexity $0 < \ka(\cdot, t) = \psi_u(\cdot, t)$ of the flow  for $ t > 0$, we can reparametrise (by inverting the tangent angle $u \mapsto \psi(u, t)$ at each time) to view the flow as a map $\ga(\psi,t)$, 
and compute that
$$\ka_u=\ka_\psi \psi_u = \ka_\psi \ka$$
and 
$$
\ka_{uu} = (\ka_u)_\psi \psi_u = (\ka_\psi \ka)_\psi \ka
= \ka_\psi^2 \ka+\ka^2 \ka_{\psi\psi}.
$$
These equations allow us to rewrite the Harnack inequality in terms of the new parametrisation as 
$$ \ka^2 \ka_{\psi\psi} +\ka^3 +\frac{\ka}{2t}\geq 0.$$
When parametrising with respect to $\psi$, 
it was computed in \cite{gage_hamilton} that the flow satisfies the equation
$$\ka_t=\ka^2 \ka_{\psi\psi}+\ka^3,$$
so the Harnack inequality becomes 
$$\ka_t +\frac{\ka}{2t}\geq 0,$$
as required.
\end{proof}

The work of Ecker-Huisken \cite{EH1} tells us that a solution in the setting of Theorem \ref{Harnack_psi} approaches the $\beta$-wedge solution in the sense that 
\beql{gamma_blowdown}
t^{-\half}\ga(\psi,t)\to \ga_\be(\psi)
\eeq
smoothly locally in $\psi$ as $t\to\infty$, where $\ga_{\be}$ is from \eqref{gamma_beta}. Therefore by integrating the Harnack inequality \eqref{Harnack_ineq_psi} from a given time $t$ to time $\infty$, we immediately obtain the following, which can also be generalised like Theorem \ref{Harnack_psi}.

\begin{thm} 
\label{classical_harnack}
In the setting of Theorem \ref{Harnack_psi}, we have 
\beql{class_harnack_ineq}
0<\ka(\psi,t)\leq t^{-\half} \ka_\be(\psi)
\eeq
for all $\psi\in (0,\pi-\be)$ and all $t\in (0,T)$, where $\ka_\be$ is from 
\eqref{kappa_beta}.
\end{thm}

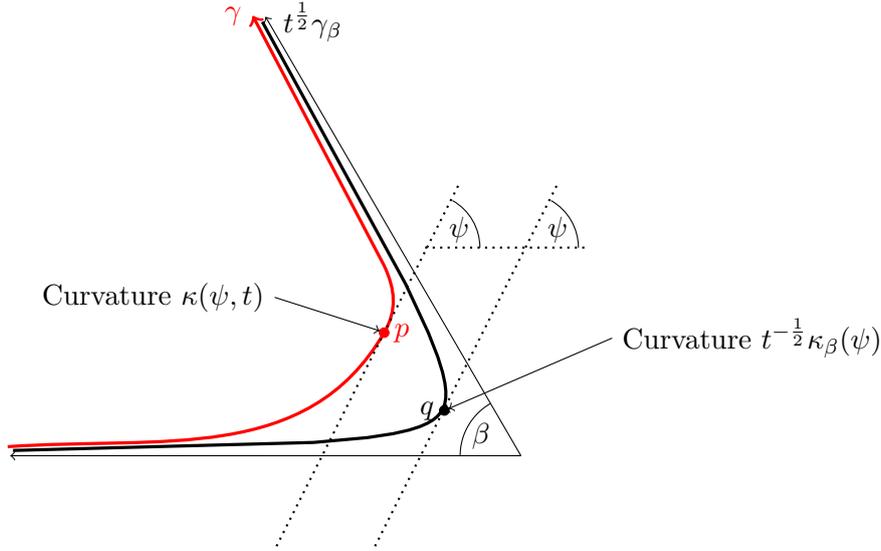
\begin{figure}
\centering
\begin{tikzpicture}[scale=1.2]
\draw [->] (0,0) -- (-2.8*1,2.8*1.732);
\draw [->] (0,0) -- (-2.8*2,0);
\draw[very thick, samples=20,domain=0.6:59.4] plot ({-cos(\x)/(sin(3*\x))^(0.5)}, {sin(\x)/(sin(3*\x))^(0.5)} ) node[right]{$\ t^\half \ga_\be$};
\draw[very thick, red, ->] (-2.8*2.01, .1) to[out=4,in=180+60] (-1.5,1.36) to[out=60, in=-65] (-1.5, 2.1) to  (-2.8*1.05, 2.8*1.7+.1) node[left]{$\ga$};
\draw[-] (-0.65,0.22) node[right]{$\be$};
\draw[samples=40,domain=0:60] plot ({-cos(\x)/1.5}, {sin(\x)/1.5}); 
\draw[thick, dotted] (-1.6,-1) -- (0.4,3);
\draw[thick, dotted] (-2.68,-1) -- (-0.68,3);
\draw[thick, dotted] (-1.05,2.3) -- (0.7,2.3);
\draw[-] (-0.45,2.3) arc (0:60:0.6);
\draw[-] (-0.9,2.5) node[right]{$\psi$};
\draw[-] (1.08-0.45,2.3) arc (0:60:0.6);
\draw[-] (1.08-0.9,2.5) node[right]{$\psi$};
\filldraw (-0.84,0.5) circle (1.5pt) node[left]{$q$};
\draw[<<-] (-0.84,0.5) -- (1,1.3) node[right]{Curvature $t^{-\half}\ka_\be(\psi)$};
\filldraw[red] (-1.5,1.36) circle (1.5pt) node[right]{$p$};
\draw[<-] (-1.55,1.38) -- (-2.7,1.75) node[left]{Curvature $\ka(\psi,t)$};
\end{tikzpicture}
\caption{Illustrating the curvature estimate \eqref{class_harnack_ineq}.}
\label{recast_Harnack}
\end{figure}

Estimate \eqref{class_harnack_ineq} is sharp because we have equality throughout when $\ga$ is the $\be$-wedge solution.
We illustrate the Harnack inequality \eqref{class_harnack_ineq} in Figure \ref{recast_Harnack}.
To develop the estimate \eqref{class_harnack_ineq} into a pointwise curvature estimate that does not involve $\beta$, 
we start by observing that 
the equality case of Corollary \ref{curv_cor}, restricted to $t=1$, gives
\beql{equality_case}
\ka_\be(\psi)\equiv  \frac{1}{2} \cd_\be(\psi),
\eeq
where $\cd_\be(\psi)$ is the 
support function of the $\be$-wedge solution.
By the comparison principle, a general solution $\ga$ as in Corollary \ref{curv_cor} will be separated from the origin by the $\be$-wedge solution $t^\half \ga_\be(\psi)$, so 
$$t^\half \cd_\be(\psi)\leq \cd(\psi,t).$$
In terms of Figure \ref{recast_Harnack}, this is just saying that the tangent line through the red point $p$ is further from the origin than the tangent line through the black point $q$. 
Thus, working from the curvature estimate \eqref{class_harnack_ineq} of Theorem \ref{classical_harnack}
we obtain 
$$\ka(\psi,t)\leq \frac{\ka_\be(\psi)}{\sqrt{t}}=  \frac{1}{2\sqrt{t}} \cd_\be(\psi) \leq  \frac{1}{2t}\cd(\psi,t),$$
which is exactly \eqref{curv_cor_est} with an alternative parametrisation. 
This connects the old and new Harnack inequalities as claimed.

\appendix
\section{Exponential decay of graphical solutions starting from half-lines}
\label{app:expdec}

In this appendix we give a simple lemma that will be used to show that curve shortening flows with initially radial ends must remain exponentially asymptotic to the ends for positive times. Since the flow will remain graphical over ends of the ends over finite time intervals, this amounts to proving estimates on solutions to \eqref{GCSF}.

\begin{lem}
\label{end_decay_lem}
Suppose that $y:[0,\infty)\times [0,T]\to\R$
is a solution to graphical curve shortening flow \eqref{GCSF} with $y(\cdot,0)\equiv 0$.
Then for all $k\in \N_0$ we have  
$$|\partial^k_x y|(x,t)\leq C(k,T) e^{-x}$$
for all $x \geq T + 3$ and all $t \in [0, T]$.
\end{lem}

\begin{proof}
We can compare the curve shortening flow given by $y$ to so-called Angenent ovals
sitting just above the $x$-axis. More precisely,  the $t$-dependent embedded closed curves consisting of points $(x,\ti y)$ satisfying
\beql{angoval2a} \sin \tilde y = 2 e^{t-a}\cosh (x-a), \quad 0 < \tilde y < \pi 
\eeq
as pictured in Figure \ref{angoval}, 
describe a curve shortening flow.
Note that  
$$1\geq \sin \tilde y =
2 e^{t-a}\cosh (x-a)\geq e^{t-a}e^{a-x}=e^{t-x}$$ 
so the oval lives within the set $x\geq t$.
As $a\to\infty$, the oval approaches the grim reaper, which satisfies
$$\sin \bar y = e^{t-x}.$$
The comparison principle forces the graph of $y$ to lie below the evolving ovals, and by taking $a\to\infty$, to therefore lie below the grim reaper. 
Thus because $ \frac{2}{\pi} \bar y  \leq \sin \bar y$ for $0 \leq \bar y \leq \frac\pi2$, we have
$$y\leq \bar y \leq \frac{\pi}{2} \sin \bar y = \frac{\pi}{2} e^{t-x}$$
for all $x\geq t$. 
By repeating the argument with an Angenent oval \emph{below} the $x$-axis we obtain the corresponding lower bound $y\geq -\frac\pi2 e^{t-x} $ also.
In particular, we have 
\beql{y_bd_both_ways}
|y|(x,t)\leq \frac{\pi}{2} e^{T-x} = C(T) e^{-x} 
\eeq 
for all $x\geq T$ and all $t\in [0,T]$.

\begin{figure}
\centering
\begin{tikzpicture}[scale = 1.2]
%
\draw[->] (0,0) -- (10.5,0) node[right]{$x$};
\draw[->] (0,0) -- (0,pi+0.5) node[above]{$\tilde y(x)$};
\draw[dashed] (0, pi) node[left]{$\pi$} -- (10.5 ,pi);
\draw[dashed] (0, pi/2) node[left]{$\displaystyle\frac{\pi}{2}$} -- (10.5 ,pi/2);
\draw[dashed] (10,0) node[below]{$2a$} -- (10,pi + .5);
\draw[dashed] (5, 0) node[below]{$a$} -- (5, pi + .5);
\draw[domain=.5:9.5, samples=50] plot  (\x, {(pi/180)*asin(2*exp(-5)*cosh(\x-5))}) ;
\draw[domain=.5:9.5, samples=50] plot  (\x, {pi - (pi/180)*asin(2*exp(-5)*cosh(\x-5))}) ;
\draw[domain=pi/10:(9/10)*pi, samples=50] plot ( { 5 - ln(   0.5*exp(5)*sin(\x*(180/pi)) +   sqrt(  ( 0.5*exp(5)*sin(\x*(180/pi)) )^2 - 1  )   )   }  , \x) ;
\draw[domain=pi/10:(9/10)*pi, samples=50] plot ( { 5 + ln(   0.5*exp(5)*sin(\x*(180/pi)) +   sqrt(  ( 0.5*exp(5)*sin(\x*(180/pi)) )^2 - 1  )   )   }  , \x) ;
\draw[very thick, domain=2.5:7.5, samples=50] plot  (\x, {(pi/180)*asin(2*exp(-3)*cosh(\x-5))}) ;
\draw[very thick, domain=2.5:7.5, samples=50] plot  (\x, {pi - (pi/180)*asin(2*exp(-3)*cosh(\x-5))}) ;
\draw[very thick, domain=pi/10:(9/10)*pi, samples=50] plot ( { 5 - ln(   0.5*exp(3)*sin(\x*(180/pi)) +   sqrt(  ( 0.5*exp(3)*sin(\x*(180/pi)) )^2 - 1  )   )   }  , \x) ;
\draw[very thick, domain=pi/10:(9/10)*pi, samples=50] plot ( { 5 + ln(   0.5*exp(3)*sin(\x*(180/pi)) +   sqrt(  ( 0.5*exp(3)*sin(\x*(180/pi)) )^2 - 1  )   )   }  , \x) ;
\draw[domain=4.3:5.7, samples=50] plot  (\x, {(pi/180)*asin(2*exp(-1)*cosh(\x-5))}) ;
\draw[domain=4.3:5.7, samples=50] plot  (\x, {pi - (pi/180)*asin(2*exp(-1)*cosh(\x-5))}) ;
\draw[domain=pi/3:(2/3)*pi, samples=50] plot ( { 5 - ln(   0.5*exp(1)*sin(\x*(180/pi)) +   sqrt(  ( 0.5*exp(1)*sin(\x*(180/pi)) )^2 - 1  )   )   }  , \x) ;
\draw[domain=pi/3:(2/3)*pi, samples=50] plot ( { 5 + ln(   0.5*exp(1)*sin(\x*(180/pi)) +   sqrt(  ( 0.5*exp(1)*sin(\x*(180/pi)) )^2 - 1  )   )   }  , \x) ;
%
%
\end{tikzpicture}
\caption{Plot of the evolving Angenent oval \eqref{angoval2a}.}
\label{angoval}
\end{figure}
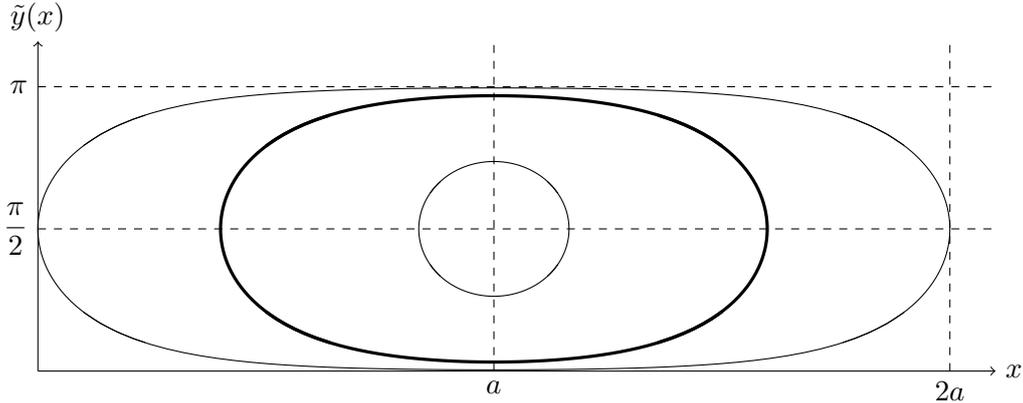

Next we require uniform gradient estimates on $y$.
We can apply the interior gradient estimates of 
Ecker-Huisken \cite[Theorem 2.3]{EH2} over $[x_0 - 1, x_0+1] \times [0, T]$, for any $x_0 \geq T + 1$ to establish that because $y(\cdot,0)\equiv 0$, there exists universal  $c_0 = c_0(T) <\infty$ such that 
 $|y_x|(x_0, t)  \leq c_0$ for all $x_0 \geq T + 1$ and $t \in [0, T]$.

For any $x_0\geq T+3$, by combining the initial condition
$y(\cdot,0)\equiv 0$ and the gradient bound $|y_x| \leq c_0$
we can apply parabolic regularity theory to the 
equation 
\beq
z_t=(\arctan z)_{xx},
\eeq
for $z:=y_x$ (cf.~\cite{ST1}) over the parabolic cylinder $[x_0-2,x_0+2]\times [0,T]$ to obtain
uniform bounds on all derivatives of $z$ over the
interior region  $[x_0-1,x_0+1]\times [0,T]$, depending only on $T$ and the order of the derivative being controlled.

Such control then allows us to view the curve shortening flow equation \eqref{GCSF} for $y$ as the linear equation 
$y_t=ay_{xx}$ with coefficient $a:=\frac{1}{1+z^2}$ controlled in every $C^k$ norm. In particular we can estimate 
$$|\partial^k_x y|(x_0,t)\leq C(k,T)\|y\|_{C^0([x_0-1,x_0+1]\times [0,T])}.$$
Recalling the bound \eqref{y_bd_both_ways}, this implies 
$$|\partial^k_x y|(x_0,t)\leq C(k,T)e^{-(x_0-1)} = C(k, T) e^{-x_0}$$
for all $ x_0 \geq T + 3$ and all $t\in [0,T]$.
\end{proof}

\vskip 0.2cm

\noindent
AS: 
\url{https://sites.google.com/view/sobnack}


\noindent
{\sc Department of Mathematics, The University of Texas at Austin, TX 78712, USA.}

%
%
%

\noindent
PT: 
\url{http://warwick.ac.uk/fac/sci/maths/people/staff/peter\_topping/} 

\noindent
{\sc Mathematics Institute, University of Warwick, Coventry,
CV4 7AL, UK.}

\end{document}